\documentclass[10pt]{amsart}
\usepackage[utf8]{inputenc}
\usepackage[T1]{fontenc}
\usepackage{amsfonts,amsmath,amssymb,amsthm}

\usepackage{float} 
\usepackage{hyperref}

\usepackage{tikz} 
\usepackage{tikz-3dplot}

\usepackage[normalem]{ulem}

\usepackage[a4paper]{geometry}
\makeatletter
\def\subsection{\@startsection{subsection}{2}%
 \z@{.5\linespacing\@plus.7\linespacing}{.3\linespacing}%
 {\normalfont\bfseries}}
\makeatother

\newcommand{\Cc}{\mathbb{C}} 

\newcommand{\Rr}{\mathbb{R}}
\newcommand{\Nn}{\mathbb{N}}

\newcommand{\Qq}{\mathbb{Q}}

\newcommand{\Kk}{\mathbb{K}}

\renewcommand{\leq}{\leqslant}
\renewcommand{\geq}{\geqslant}
\renewcommand{\le}{\leqslant}

\newcommand{\ii}{\mathrm{i}}

\renewcommand{\epsilon}{\varepsilon}
\newcommand{\defi}[1]{\textbf{#1}}

\DeclareMathOperator{\Crit}{Crit}
\DeclareMathOperator{\Int}{\int}
\DeclareMathOperator{\lc}{\chi}

\DeclareMathOperator{\val}{\nu}
\DeclareMathOperator{\sgn}{sgn}
\newcommand{\order}{<} 
\newcommand{\dd}{\mathrm{d}} 
\newcommand{\hot}{\emph{hot}}

{\theoremstyle{plain}
\newtheorem{theorem}{Theorem}[section] % theorem with number
\newtheorem*{theoremA}{Theorem A}            % theorem without number
\newtheorem*{theoremB}{Theorem B}            % theorem without number
\newtheorem{lemma}[theorem]{Lemma} % lemma with number
\newtheorem{proposition}[theorem]{Proposition} % lemma with number
}

{\theoremstyle{remark}
\newtheorem{definition}[theorem]{Definition} 

\newtheorem*{remark*}{Remark} % theorem without number
\newtheorem{remark}[theorem]{Remark} 
\newtheorem{example}[theorem]{Example}
}

\newcommand{\sauteligne}{\leavevmode}

\newcommand{\myfigure}[2]{% entrée : échelle, fichier(s) figure à inclure
\begin{center}\small
\tikzstyle{every picture}=[scale=1.0*#1]% mise en échelle 
#2
\end{center}}

\title[Morsifications of real singularities]{Combinatorial study of morsifications \\ of real univariate singularities}

\author[Bodin]{Arnaud Bodin}
\author[García Barroso]{Evelia Rosa García Barroso}
\author[Popescu-Pampu]{Patrick Popescu-Pampu}
\author[Sorea]{Miruna-\c Stefana Sorea}

\email{arnaud.bodin@univ-lille.fr}
\email{ergarcia@ull.es}
\email{patrick.popescu-pampu@univ-lille.fr}
\email{mirunastefana.sorea@ulbsibiu.ro}

\address{(Arnaud Bodin) Universit\'e de Lille, CNRS, Laboratoire Paul Painlev\'e, 59000 Lille, France}
\address{(Evelia Rosa García Barroso) Universidad de La Laguna. IMAULL.
Departamento de Matemáticas, Estadística e I.O., Apartado de Correos 456.
38200, La Laguna, Tenerife, España}
\address{(Patrick Popescu-Pampu) Universit\'e de Lille, CNRS, Laboratoire Paul Painlev\'e, 59000 Lille, France}
\address{(Miruna-\c Stefana Sorea) Lucian Blaga University of Sibiu, 550024 Sibiu,  Romania}


\thanks{\emph{Acknowledgments}. 
The authors gratefully acknowledge the support of Universidad de La Laguna (Tenerife, Spain), where part of this work was done (Spanish grant PID2019-105896GB-I00 funded by
MCIN/AEI/10.13039/501100011033). This work was also supported by the Labex CEMPI (ANR-11-LABX-0007-01) and ANR SINTROP (ANR-22-CE40-0014). M.-{\c S}. Sorea was supported by the project ``Mathematical Methods and Models for
Biomedical Applications'' financed by National Recovery and Resilience Plan PNRR-III-C9-2022-I8. 
M.-{\c S}. Sorea is grateful to Antonio Lerario for the very supportive working environment during her three-year postdoc at SISSA (Scuola Internazionale Superiore di Studi Avanzati), Trieste, Italy. We thank Erwan Brugall\'e and Christopher-Lloyd Simon for their remarks 
and the anonymous referee for valuable feedback and suggestions.}

\date{\today}

\begin{document}

\begin{abstract}
We study a broad class of morsifications of germs of univariate real analytic functions. We characterize the combinatorial types of the resulting Morse functions via planar contact trees constructed from Newton-Puiseux roots of the polar curves of the morsifications. 
\end{abstract}
 
\maketitle

%%%%%%%%%%%%%%%%%%%%%%%%%%%%%%%%%%%%%%%%%%%%%%%%
\section{Introduction}
\label{sect:intro}

%-----------------------------------------------
\subsection{Morsifications}
 
In this paper, by a \emph{singularity} we mean a germ of real or complex analytic function with an isolated critical point. By a \emph{Morse function} on a compact manifold with boundary we mean a smooth function having only non-degenerate critical points, all of them interior to the manifold, and pairwise distinct critical values. 
A powerful method for analyzing a singularity is to deform it in a suitable way and relate it to the various resulting simpler singularities. This method has been extensively used for complex singularities. For instance, a generic holomorphic deformation of a complex singularity with Milnor number $\mu$ produces a Morse function with exactly $\mu$ critical points (see \cite[page 150]{ebeling}). However, similar  \emph{morsifications} of real singularities have been much less studied, even in the case of one variable. 

\emph{In this paper we examine the combinatorial types of morsifications of univariate real singularities.} This problem is inspired by Arnold's papers \cite{snakes}, \cite{CN}, which studied the combinatorial types of real Morse univariate polynomials, and by Ghys' book \cite{ghys_promenade}, which examined the combinatorial types of real plane curve singularities (see also Ghys' paper \cite{ghysIntersecting} and Ghys and Simon's paper \cite{ghysSimon}). Let us mention also two very recent related articles. In \cite{teissier_clarif}, Teissier describes some open problems about real morsifications and in \cite{vassiliev}, Vassiliev describes the possible real Morsification types in the case of a real simple singularity in any number of variables.

%-----------------------------------------------
\subsection{Bi-ordered critical sets as measures of the combinatorial types of morsifications}
\label{ssec:biordcrit}

 We encode the combinatorial type of a Morse function defined on a compact interval by a \emph{bi-ordered set}: its critical set endowed with the total order induced by its inclusion in the source interval and with the total order of the corresponding critical values.
Let $F_0(y) \in \Rr\{y\}$ be a convergent power series defining a univariate real singularity. Fix a compact interval $I$ around $y=0$ on which $F_0(y)$ is defined and has a single critical point at the origin. Let $F_x(y) \in \Rr\{x,y\}$ be a morsification of $F_0(y)$. This means that for every small enough $x_0>0$, the functions $F_{x_0} : I \to \Rr$ are Morse and have the same combinatorial type. Moreover, this combinatorial type is independent of the choice of interval $I$, it is therefore canonically attached to the morsification $F_x(y) \in \Rr\{x,y\}$. 

\emph{Our central problem is to compute this combinatorial type starting from the series $F_x(y)$.}

We solve this problem under a suitable hypothesis, the \emph{injectivity condition}. Our answer is governed by the \emph{contact tree} $T_{\Rr}(f)$ of the \emph{real} Newton-Puiseux roots $\xi_i$ of $f(x,y) := \partial_y F_x(y)$. It is a rooted planar tree whose leaves correspond bijectively to the series $\xi_i$. As an abstract tree, it is determined by the \emph{valuations} of the pairwise differences of those series, that is, by the initial exponents of those differences. In turn, its planar structure is given by the total order $\order_{\Rr}$ on its set of leaves such that $\xi_i \order_{\Rr} \xi_j$ if and only if $\xi_i(x_0) < \xi_j(x_0)$ for $x_0 > 0$ small enough.

Under the injectivity condition, we construct canonically from $F_x(y)$ a second planar structure on the abstract rooted tree $T_{\Rr}(f)$. This second planar structure determines a new total order on the set of leaves of $T_{\Rr}(f)$. In Theorem A, we prove that: 

\begin{theoremA} 
\label{th:A}
Assume that $f$ satisfies the injectivity condition. Then, for $x_0>0$ small enough, the bi-ordered critical sets of the Morse functions $F_{x_0} : I \to \Rr$ are isomorphic to the set of leaves of the contact tree $T_{\Rr}(f)$, endowed with the total orders determined by the two planar structures above.
\end{theoremA}

The previous result led us to ask whether $T_{\Rr}(f)$, endowed with its second planar structure, may also be interpreted as a contact tree. In Theorem B we prove that this is indeed the case: 

\begin{theoremB} 
\label{th:B}
Assume that $f$ satisfies the injectivity condition. Then, with its second planar structure, $T_{\Rr}(f)$ is isomorphic to the contact tree of the real Newton-Puiseux roots of the \emph{discriminant curve} of the morphism $(x,y) \mapsto (x, F_x(y))$.
\end{theoremB}

This discriminant curve is the \emph{critical image} of this morphism, also called \emph{apparent contour in the target}. The \emph{apparent contour in the source} is the curve $f(x,y) =0$, that is, the \emph{polar curve} of $F_x(y)$ relative to $x$. 

Let us assume that the injectivity condition is satisfied. Then, as a consequence of Theorem A, the structure of the real contact tree $T_{\Rr}(f)$ strongly constrains the combinatorial type of the morsifications $F_x(y)$ (see Remark \ref{rem:strongconstr}) and,  as a consequence of Theorem B, the real contact trees of the apparent contours in the source and in the target of the morphism $(x,y) \mapsto (x, F_x(y))$ are isomorphic as abstract rooted trees (see Remark \ref{rem:absisom}).

%-----------------------------------------------
\subsection{The meaning of the injectivity condition}
\label{ssec:meaninginj}

 Assume that the real Newton-Puiseux roots of $f$ are numbered such that $\xi_1 \order_{\Rr} \cdots \order_{\Rr} \xi_n$. The bi-ordered critical set of a Morse function $F_{x_0} : I \to \Rr$, for small enough $x_0>0$, is determined by the signs of all the differences $F_{x_0}(\xi_j(x_0)) - F_{x_0}(\xi_i(x_0))$ of its critical values. If $i < j$, we may write:
 $$ F_x(\xi_j) - F_x(\xi_i) = S_i + \cdots + S_{j-1}, $$
where $S_r := F_x(\xi_{r+1}) - F_x(\xi_r)$. For small enough $x_0 > 0$, the sign of $F_{x_0}(\xi_j(x_0)) - F_{x_0}(\xi_i(x_0))$ is thus equal to the sign of the initial coefficient of the sum  
$S_i + \cdots + S_{j-1}$ of real Newton-Puiseux series. 
We meet the precise situation in which Newton introduced the method of turning ruler, which led to the notion of Newton polygon (see \cite[pages 51--53]{ghys_promenade}): denoting by $\val(S_l)$ the valuation of $S_l$, we know that the initial coefficient of the 
sum $S_i + \cdots + S_{j-1}$ is the sum of the initial coefficients $s_r$ of the series $(S_r)_{i \leq r < j}$ achieving the minimum 
 $$ \min \{ \val(S_r) \mid i \leq r < j\},$$
\emph{provided that this last sum of initial coefficients $s_r$ is non-zero}. The \emph{injectivity condition} of Definition \ref{def:injcond} is equivalent to the fact that this non-vanishing condition is satisfied for every pair $(i,j)$ with $i < j$.

%-----------------------------------------------
\subsection{Structure of the paper}
 
As the function which controls the combinatorial types of the Morse functions $y \mapsto F_{x_0}(y)$ is $f(x,y):= \partial_y F_x(y)$ rather than $F_x(y)$, we prefer to start from a real analytic series $f(x,y)$ and integrate it relative to $y$ in order to get the series $F_x(y)$. 
In Section \ref{sec:real-branches} we recall the factorization of $f(x,y)$ via its Newton-Puiseux roots, we distinguish between real and non-real roots and we define the notions of \emph{right-reduced series} and of \emph{primitive} of $f(x,y)$. 
In Section \ref{sec:right-morse} we explain the needed notions about \emph{univariate Morse functions} and their \emph{bi-ordered critical sets}, as well as about \emph{morsifications} of univariate singularities. In Section \ref{sec:contact-trees} we explain basic facts about \emph{rooted} and \emph{planar} trees and we introduce the types of rooted trees used in the paper:  the \emph{real contact tree} $T_{\Rr}(f)$ mentioned above, and the \emph{complex contact tree} $T_{\Cc}(f)$, which is an abstract rooted tree containing $T_{\Rr}(f)$.  Section \ref{sec:injectivity} contains our main technical results. In it, we introduce the \emph{area series} $S_l$ mentioned in Subsection \ref{ssec:meaninginj}, we compute their valuations in terms of the embedding $T_{\Rr}(f) \hookrightarrow T_{\Cc}(f)$ (see Proposition \ref{prop:sigmai}) and we deduce the valuations of the sums $S_i + \cdots + S_{j-1}$ under the non-vanishing hypothesis mentioned in Subsection \ref{ssec:meaninginj} (see Lemma \ref{lem:diffcrit}).
In Section \ref{sec:theorem} we define the \emph{injectivity condition} (\ref{Inj}), we give examples in which it is not satisfied (see Examples \ref{ex:symtree} and \ref{ex:notnec}), we prove our first main result, Theorem A, and we explain that Lemma \ref{lem:diffcrit} allows to get a weaker statement even if the injectivity condition is not satisfied (see Remark \ref{rem:notinj}). In Section \ref{sec:contour} we define real polar and discriminant curves and we prove our second main result, Theorem B. We conclude the paper by an example with parameters, explained in Section \ref{sec:example}.

%-----------------------------------------------
\subsection{Related works}

In this paper, we generalize results of the PhD thesis \cite{sorea2018shapes} of the last author, 
published in \cite{sorea_portugaliae}, \cite{sorea_SymbComp}, \cite{sorea_fourier}.
In those works, the polar curve $f(x,y)=0$ and the series $F_x(y)$ had to respect some hypotheses: 

-- all the branches of the polar curve were real, distinct, smooth and transverse to the vertical axis $x=0$; 

-- the real contact tree $T_{\Rr}(f)$ was a rooted binary tree (then the  injectivity condition is  automatically satisfied); 

-- $F_x(y)$ had a strict local minimum at $(0,0)$.

The aim was to describe the asymptotic shape of the level curves $F_x(y)=\varepsilon$ when $\varepsilon>0$ converged to $0$. This description was done in terms of a \emph{Poincaré-Reeb tree} measuring the non-convexity of the interior of the topological disk bounded by the level curve $F_x(y)=\varepsilon$, relative to the direction $x$ (see also \cite{BPS} for a general study of level curves of real bivariate polynomials). 
Here we replace all the former hypotheses by two much less restrictive conditions, namely that:

-- the real Newton-Puiseux roots $\xi_i$ of $f(x,y)$ are pairwise distinct; 

-- the \emph{injectivity condition} is satisfied.

%-----------------------------------------------
\subsection{An explanatory picture} 
\label{ssec:fundpict}

\begin{figure}[H]
\myfigure{0.85}{
	\tdplotsetmaincoords{70}{12}
\begin{tikzpicture}[tdplot_main_coords,scale=3]

%%%%%%%%%%%%%%%%%%%%%%%
% Projection xy
\begin{scope}[yshift=-1.75cm]

\draw[fill=gray!10,opacity=0.5] (-1,2,0) -- (-1,-2,0) -- (2,-2,0) -- (2,2,0) -- cycle;
\draw[thick,->,>=latex] (0,0,0) -- (1.85,0,0) node[below]{$x$};
\draw[thick,->,>=latex] (0,-1.2,0) -- (0,1.5,0) node[above]{$y$};

\draw[ultra thick, color=blue,domain=-1.2:1.2,samples=50,smooth] plot ({(\x)^2},{\x}) node[above,scale=0.9,black] {$f(x,y)=0$} ;
\node[blue] at (1.55,-1.3) {$\xi_1$};
\node[blue] at (1.55,1.15) {$\xi_2$};

\draw[thick,green!70!black] (1,-1.3,0) -- (1,1.3,0) node[midway,below right]{$x_0$};

\end{scope}

%%%%%%%%%%%%%%%%%%%%%%%
% 3D picture
\begin{scope}

\draw[-stealth,red!50,line width=3pt,opacity=1] (-0.5,0,-0.5) -- ++(0,0,-1);
\draw[fill=gray!10,opacity=0.5] (-1,2,0) -- (-1,-2,0) -- (2,-2,0) -- (2,2,0) -- cycle;

\draw[thick,->,>=latex] (0,0,0) -- (1.85,0,0) node[below]{$x$};
\draw[thick,->,>=latex] (0,-1.2,0) -- (0,1.5,0) node[above]{$y$};
\draw[thick,->,>=latex] (0,0,0) -- (0,0,1.05) node[left]{$z$};

\begin{scope}[canvas is yz plane at x=1,scale=0.5]
\draw[fill=red!5,opacity=0.5] (-4,-2.5) rectangle ++(8,5);
\def\c{1.0}
\def\x{-1.74}
\coordinate (P1) at (2*\x,{\x^3-3*\c*\x)});

\def\x{-1}
\coordinate (P2) at (2*\x,{\x^3-3*\c*\x)});

\def\x{1}
\coordinate (P3) at (2*\x,{\x^3-3*\c*\x)});

\def\x{1.83}
\coordinate (P4) at (2*\x,{\x^3-3*\c*\x)});

\end{scope}

\begin{scope}[canvas is yz plane at x=1,scale=0.5]
%\draw[fill=red!5,opacity=0.5] (-4,-2.5) rectangle ++(8,5);
\def\c{1.0}

\draw[ultra thick, color=green!70!black,domain=-1.74:0,samples=50,smooth] plot (2*\x,{\x^3-3*\c*\x)});

\draw[ultra thick, color=green!70!black!40,domain=-0:1.75,samples=50,smooth] plot (2*\x,{\x^3-3*\c*\x)});
 \draw[ultra thick, color=green!70!black,domain=1.75:1.83,samples=50,smooth] plot (2*\x,{\x^3-3*\c*\x)}) node[right,green!70!black]
 {$F_{x_0}(y)$};

\node[scale=3] at (0,0) {.};

\end{scope}

\draw[thick] (1,0,0) -- ++ (0.5,0,0);

\end{scope}

%%%%%%%%%%%%%%%%%%%%%%%
% Projection xz
\begin{scope}[xshift=0.4cm, yshift=1.95cm]

\draw[fill=gray!10,opacity=0.5] (-1,0,1) -- (-1,0,-1) -- (2,0,-1) -- (2,0,1) -- cycle;
\draw[thick,->,>=latex] (0,0,0) -- (1.85,0,0) node[below]{$x$};
\draw[thick,->,>=latex] (0,0,0) -- (0,0,0.90) node[left]{$z$};

\begin{scope}[canvas is xz plane at y=0,scale=0.5]
\draw[ultra thick, color=red!80,domain=0:1.4,samples=50,smooth] plot ({\x^2},{0.5*\x^3}) node[right] {$\delta_1$};
\draw[ultra thick, color=red!80,domain=0:1.4,samples=50,smooth] plot ({\x^2},{-0.5*\x^3}) node[right] {$\delta_2$};
\end{scope}

\end{scope}

\draw[-stealth,red!50,line width=3pt,opacity=1] (-0.85,1.5,0.05) -- ++(0,2.1,0) node[midway,left]{$\pi$};

\begin{scope}[canvas is yz plane at x=0,scale=0.5]
\draw[ultra thick, color=violet!90!black!40,domain=-0.91:0,samples=50,smooth] plot (2*\x,{\x^3)});
\def\x{-0.9}
\coordinate (P5) at (2*\x,{\x^3)});
\draw[ultra thick, color=violet!90!black,domain=0:1.05,samples=50,smooth] plot (2*\x,{\x^3)});
\def\x{1.05}
\coordinate (P6) at (2*\x,{\x^3)});
\draw[fill=red!5,opacity=0.5] (-4,-2.5) rectangle ++(8,5);
\end{scope}
\draw (P5) to[bend left=4] (P1);
\draw[ultra thick,orange] (0,0,0) to[bend right=10] (P2);
\draw[ultra thick,orange!60] (0,0,0) to[bend left=15] (P3);
\draw (P6) to[bend right=3] (P4);
\end{tikzpicture}%

}
\caption{The graph of a morsification $(x,y) \mapsto F_x(y)$} of $y \mapsto y^3$, its source and target projections and sections of the graph by the planes defined by $x = 0$ and $x = x_0$.
\label{fig:fundamental}
\end{figure}
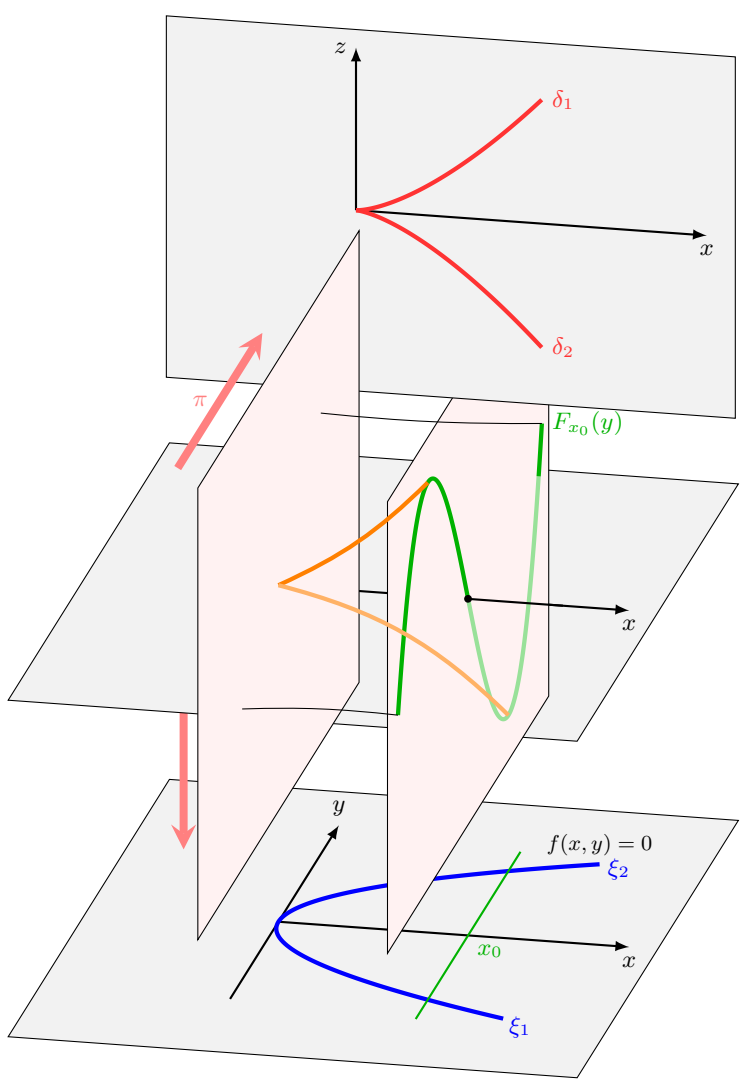
 
Figure \ref{fig:fundamental} introduces the main geometric objects studied in this paper.
It corresponds to Whitney's classical \emph{cusp singularity} from \cite{whitney} of a map between real planes.
This example will also illustrate our main Theorems A and B (see Examples \ref{ex:parabola1} and \ref{ex:parabola2}).
We start from the real plane curve germ $f(x,y)=0$ at $(0,0)$ represented in the real plane $\Rr^2_{x,y}$ at the bottom, 
where $f(x,y) = 3(y^2-x) = 3 (y+x^{\frac{1}{2}})(y-x^{\frac{1}{2}})$. It has two Newton-Puiseux roots $\xi_1 = -x^{\frac{1}{2}}$ and $\xi_2 = x^{\frac{1}{2}}$. Both are real. 
We define $F_x(y)$ as a primitive of $f$ w.r.t.{} the variable $y$. Here we choose $F_x(y) := y^3-3xy$.
The graph of the function $(x,y) \mapsto F_x(y)$, for positive $x$, is the surface depicted in the central part of the figure. 
By intersecting this surface with a vertical plane defined by $x=x_0$, we get the graph of $y \mapsto F_{x_0}(y)$. In our example, it is a Morse function with one local maximum and one local minimum, for every $x_0 >0$. It is possible to follow these local extrema when $x_0$ tends to $0$: they trace the two orange curves on the surface. These two curves project to the real plane $\Rr^2_{x,y}$ exactly onto the graphs of the roots $\xi_1$ and $\xi_2$.

On the other hand, these orange curves form the apparent contour in the source of the projection $\pi$ of the surface above onto the vertical plane $\Rr^2_{x,z}$. The apparent contour in the target plane $\Rr^2_{x,z}$ is the \emph{discriminant curve} of $\pi$, which in this example consists of the graphs of two real Newton-Puiseux series $\delta_1$ and $\delta_2$ (at the top of Figure \ref{fig:fundamental}). The graphs of $\xi_1$ and $\xi_2$ in the real plane $\Rr^2_{x,y}$ form the \emph{polar curve} of $F_x(y)$ with respect to $x$; it is defined by the equation $f(x,y) =0$. By the projection $\pi$, the series $\xi_1$ corresponds to $\delta_1$, and $\xi_2$ to $\delta_2$. In the plane $\Rr^2_{x,y}$, $\xi_1$ appears \emph{before} $\xi_2$ (we will define the real total order on the ring of real Newton-Puiseux series in Section \ref{sec:real-branches}), while in the plane $\Rr^2_{x,z}$, $\delta_1$ appears \emph{after} $\delta_2$. This permutation 
$\left(\begin{smallmatrix}
	1 & 2  \\
	2 & 1 \\
\end{smallmatrix}\right)$ 
encodes the combinatorial type of the Morse function $y \mapsto F_{x_0}(y)$.
Theorem A explains that whenever the \emph{injectivity condition} is satisfied, the corresponding permutation may be read from the embedding $T_{\Rr}(f) \hookrightarrow T_{\Cc}(f)$. 

The results of our paper allow therefore to make pictures analogous to that of Figure \ref{fig:fundamental}, representing correctly the combinatorial types of the Morse functions $y \mapsto F_x(y)$ whenever $f$ satisfies the injectivity condition.

%%%%%%%%%%%%%%%%%%%%%%%%%%%%%%%%%%%%%%%%%%%%%%%%
\section{Real Newton-Puiseux series, right semi-branches and primitives}
\label{sec:real-branches}

In this section we explain our notations about \emph{Newton-Puiseux series}, we define \emph{right semi-branches} as the germs of graphs of real Newton-Puiseux series and we introduce the notion of \emph{primitive} of a bivariate series.

%-----------------------------------------------
\subsection{Newton-Puiseux series}
\label{ssec:NPs}

For $\Kk=\Cc$ or $\Rr$, let $\Kk\{x\}$ and $\Kk\{x,y\}$ denote the ring of convergent power series in one and two variables respectively, with coefficients in the field $\Kk$. Consider also the ring 
   \[ \Kk\{x^{\frac{1}{\Nn}}\} := \big\{g(x^\frac{1}{n}) \mid g\in \Kk \{t\}, n\in\Nn^*\big\} \]
of \defi{Newton-Puiseux series in the variable $x$}, with coefficients in $\Kk$. Then $\Rr\{x^{\frac{1}{\Nn}}\}\subset\Cc\{x^{\frac{1}{\Nn}}\}$. 

Let $\gamma\in\Cc\{x^{\frac{1}{\Nn}}\}\setminus\{0\}$. We may write uniquely:
   \[ \gamma = s x^\sigma + \hot, \] 
such that $s\in\Cc^*$, $\sigma\in \Qq \ \cap \ [0,+\infty)$ and the remainder $\hot$ (which stands for \emph{higher order terms}) gathers the terms of $\gamma$ whose exponents are greater than $\sigma$. The number $s\in\Cc^*$ is the \defi{initial coefficient} of $\gamma$, denoted by $\lc(\gamma)$, and $\sigma\in\mathbb{Q} \ \cap \ [0,+\infty)$ is the \defi{initial exponent} of $\gamma$, denoted by $\val(\gamma)$. By convention, $\val(0)=\infty$. The function 
  \[ \val: \Cc\{x^{\frac{1}{\Nn}}\} \to \mathbb{Q} \cup \{\infty\} \] 
is a \emph{ring valuation}, which will play a crucial role in the sequel.

The ring $\Rr\{x^{\frac{1}{\Nn}}\}$ of real Newton-Puiseux series is naturally totally ordered:
\begin{definition}
 \label{def:realtot}
 The \defi{real total order} $\order_{\Rr}$ on the ring $\Rr\{x^{\frac{1}{\Nn}}\}$ is defined as follows: for any two distinct Newton-Puiseux series $\xi_1, \xi_2\in \Rr\{x^{\frac{1}{\Nn}}\}$, $\xi_1 \order_{\Rr} \xi_2$ if and only if $\lc (\xi_2-\xi_1)>0$. 
\end{definition}

Note that $\xi_1 \order_{\Rr} \xi_2$ if and only if $\xi_1(x_0) < \xi_2(x_0)$ for $x_0 > 0$ small enough.

%-----------------------------------------------
\subsection{Right semi-branches}
\label{ssec:rightsb}

Consider $\xi\in \Rr\{x^{\frac{1}{\Nn}}\}$ with $\xi(0) = 0$. In the sequel it will be often needed to turn the formal series $\xi$ into a real-valued function. This will be performed by choosing a real number $\epsilon\in (0, \infty)$ such that the series with real terms $\xi(x_0)$ converges for every $x_0 \in [0, \epsilon]$. For simplicity, we still denote by $\xi: [0, \epsilon] \to \Rr$ the resulting function. We will say that it is \defi{the sum of the series $\xi$}.
The sum of the series $\xi$ depends on the chosen interval of convergence $[0, \epsilon]$, but its germ at the origin is well-defined. 
Therefore, the germ at $(0,0)\in\Rr^2$ of the graph 
$$\Gamma_{\xi} := \big\{(x_0,\xi(x_0))\mid x_0 \in [0,\epsilon]\big\}$$ 
of the function $\xi$ is also well-defined. We call it the \defi{right semi-branch of the series $\xi \in \Rr\{x^{\frac{1}{\Nn}}\}$}.

If $f\in \Rr\{x,y\}$ is such that $f(0, 0) =0$ but $f(0,y) \not\equiv 0$, then by the Weierstrass preparation theorem (see \cite[page 107]{fischer}), together with the Newton-Puiseux theorem (see \cite[Theorem 1.2.20]{handbookCurves}, \cite[Section 8.3]{brieskorn}), we can write in a unique way:
\begin{equation}
\label{eq:WeierstrassPrepThm}
 f(x,y)=u(x,y)\prod_{i=1}^k (y-\gamma_i),
\end{equation}
such that $u \in \Cc\{x,y\}$ is a unit (i.e.{} $u(0,0) \neq 0$) and $\gamma_i \in \Cc\{x^{\frac{1}{\Nn}}\}$ for all $i \in \{1, \dots, k\}$. Since $f\in \Rr\{x,y\}$, we have that $u\in \Rr\{x,y\}$.
The Newton-Puiseux series $\gamma_i$ are called the \defi{Newton-Puiseux roots} of $f$. 
We will make below (see Formula (\ref{eq:xiAndEta})) a distinction between roots having only real coefficients (denoted by $\xi_i$) and the others (denoted by $\eta_l$).
We denote by $\mathcal{R}_{\Kk}(f)$ the \emph{multi-set} of roots $\gamma\in \Kk\{x^{\frac{1}{\Nn}}\}$ of $f$ (that is, each root is counted with its multiplicity). 
The \defi{set of right semi-branches of} $f$ is by definition the set of right semi-branches of the elements of $\mathcal{R}_{\Rr}(f)$.

%-----------------------------------------------
\subsection{Primitives}
\label{ssec:Primit}

Consider $f \in \Rr\{x,y\}$ with $f(0,0) = 0$ and $f(0, y) \not\equiv 0$. A \defi{primitive} of $f$ is a series $F_x(y)\in \Rr\{x,y\}$ such that: 
\begin{equation}
\label{eq:PrimitiveOff}
    \partial_y F_x = f.
\end{equation}
Primitives of $f$ always exist. They are of the form $g(x) + h(x,y)$,
where $g \in \Rr\{x\}$ is arbitrary and $h \in \Rr\{x,y\}$ is obtained by termwise integration of the series $f$, that is, by replacing each non-zero term $c_{p,q} \cdot x^p \cdot y^q$ of it by $c_{p,q}\cdot x^p \cdot \dfrac{y^{q+1}}{q+1}$.

%%%%%%%%%%%%%%%%%%%%%%%%%%%%%%%%%%%%%%%%%%%%%%%%
\section{Morsifications}
\label{sec:right-morse}

In this section we give basic vocabulary about univariate \emph{Morse functions} and we introduce their \emph{bi-ordered critical graphs}. Then we define \emph{morsifications} of univariate singularities and their \emph{combinatorial types}.

%-----------------------------------------------
\subsection{Morse functions}
 
Let us first introduce standard definitions from Morse theory, particularized to our context of univariate functions:

\begin{definition} 
\label{def:Morsevocab}
Let $I \subset \Rr$ be a compact interval and let $\varphi : I \to \Rr$ be a smooth function. 
We say that $c \in I$ is a \defi{critical point} of $\varphi$ if $\varphi'(c)=0$; it is called  \defi{non-degenerate} if $\varphi''(c) \neq 0$. 
We say that $\varphi$ is a \defi{Morse function} if:

-- all its critical points are non-degenerate;

-- they lie in the interior of $I$; 

-- its critical values are pairwise distinct. 

The \defi{critical graph} of $\varphi$ is the graph of the restriction of $\varphi$ to its set of critical points: 
\[\Crit(\varphi) := \big\{(c,\varphi(c))\mid c \text{ is a critical point of } \varphi\big\}.\] 
\end{definition}

Non-degenerate critical points being isolated, a Morse function on a compact interval has only a finite number of critical points. 
In the literature, what we call Morse functions are sometimes called \emph{excellent} Morse functions, the attribute referring to the third condition above, which is equivalent to the condition that no two critical points lie on the same level set. As we do not consider non-excellent Morse functions, we prefer to use the simplified terminology of Definition \ref{def:Morsevocab}.

%-----------------------------------------------
\subsection{The canonical bi-order on the critical graph of a Morse function}\label{sec:snakes}

In this paper, by an \defi{order} we mean either a strict or non-strict partial or total order on a given set, depending on the context. We will denote by $\prec$ the strict order associated to an order $\preceq$.
 
A finite set $\mathcal{S}$ is \defi{bi-ordered} if it is endowed with a pair of total orders. The critical graph $\Crit(\varphi)$ (see Definition \ref{def:Morsevocab}) of any Morse function $\varphi:I \to \Rr$ defined on a compact interval $I$ is canonically bi-ordered:

\begin{definition} 
\label{def:canbiordmorse}
Let $I \subset \Rr$ be a compact interval and $\varphi:I \to \Rr$ be a Morse function. The \defi{source order} $\order_s$ and \defi{target order} $\order_t$ are the total orders on the critical graph $\Crit(\varphi)$ defined as follows for any two distinct points $p=(y_1,z_1), q=(y_2,z_2) \in \Crit(\varphi)$:
\begin{itemize}
 \item $p \order_s q$ if and only if $y_1 < y_2$, 
 \item $p \order_t q$ if and only if $z_1 < z_2$. 
\end{itemize} 
The \defi{canonical bi-order} on $\Crit(\varphi)$ is the pair $(\order_s, \order_t)$.
\end{definition}

\begin{remark}
\sauteligne
\begin{enumerate}
  \item The bi-ordered set $(\Crit(\varphi),\order_s,\order_t)$ may be thought as a measure of the combinatorial type of the Morse function $\varphi$. 
 Indeed, let $\varphi_1 : I_1 \to \Rr$ and $\varphi_2 : I_2 \to \Rr$ be two Morse functions on compact intervals. Then the associated bi-ordered critical sets are isomorphic if and only if the restrictions of $\varphi_1$ and $\varphi_2$ to the minimal intervals containing all their critical points are right-left equivalent by orientation-preserving diffeomorphisms. Without restricting $\varphi_1$ and $\varphi_2$ in this way, one should also take into account their boundary values in order to construct a complete invariant of right-left equivalence.

 \item As explained in \cite[pages 17--18]{ghys_promenade} (see also \cite[Section 3.2.6]{sorea2018shapes}), the comparison between the two total order relations on a bi-ordered set naturally gives rise to a permutation. 
The permutations coming from Morse functions were called \emph{snakes} by Arnold (see \cite{snakes}, \cite{CN}, \cite[Definition 1.4]{sorea_portugaliae}). We will use again this terminology in Section \ref{sec:example} (see Figure \ref{fig:snakeThreeCusps}).
\end{enumerate}
\end{remark}

\begin{example}
\label{ex:snake4}
Let us consider the Morse function $y \mapsto z = \varphi(y)$ whose graph is pictured in Figure \ref{fig:biorder}. 

\begin{figure}[H]
\myfigure{0.9}{%
	\begin{tikzpicture}[scale=0.7]
	
% Axes
\draw[->,>=latex,thick, gray] (-0.5,0)--(6,0) node[below,black] {$y$};
\draw[->,>=latex,thick, gray] (0,-0.5)--(0,5) node[left,black] {$z$};

\draw[looseness=0.5,green!70!black, very thick] (0.1,0.25) to[out=80,in=180] 
(1,3) to[out=0,in=180] 
(2,1) to[out=0,in=180] 
(3,4) to[out=0,in=180] 
(4,2) to[out=0,in=-120] 
(5,4.75) node[right,black]{$z=\varphi(y)$}
;

\fill[orange] (1,3) circle (3pt) node[above,black]{$p_1$};
\fill[orange] (2,1) circle (3pt) node[below,black]{$p_2$};
\fill[orange] (3,4) circle (3pt) node[above,black]{$p_3$};
\fill[orange] (4,2) circle (3pt) node[below,black]{$p_4$};

\fill[blue] (1,0) circle (1pt) node[below]{$1$};
\fill[blue] (2,0) circle (1pt) node[below]{$2$};
\fill[blue] (3,0) circle (1pt) node[below]{$3$};
\fill[blue] (4,0) circle (1pt) node[below]{$4$};

\fill[green!70!black] (0,1) circle (1pt) node[left]{$1$};
\fill[green!70!black] (0,2) circle (1pt) node[left]{$2$};
\fill[green!70!black] (0,3) circle (1pt) node[left]{$3$};
\fill[green!70!black] (0,4) circle (1pt) node[left]{$4$};

\draw [->,>=latex,blue!20, line width=3pt] (0.5,-1.2) -- ++(4.25,0) node[midway,below,black]{order $<_s$};
\draw [->,>=latex,green!40, line width=3pt] (-1.4,0.5) -- ++(0,4) node[midway,left,black]{order $<_t$};
\end{tikzpicture}%
}
\caption{A Morse function and the two total orders on its critical set.}
\label{fig:biorder}
\end{figure}
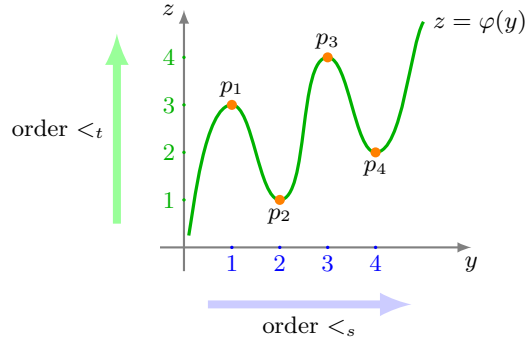

Its critical graph $\Crit(\varphi)$ has $4$ elements $p_1, \dots , p_4$.
Since $y_1 < y_2 < y_3 < y_4$ and $z_2 < z_4 < z_1 < z_3$, the two orders on it are :
$$p_1 \order_s p_2 \order_s p_3 \order_s p_4
\qquad \text{ and } \qquad
p_2 \order_t p_4 \order_t p_1 \order_t p_3.
$$

The associated snake is 
$$\pi_\varphi = \begin{pmatrix}
			{\color{blue} 1 } & {\color{blue} 2} & {\color{blue} 3} & {\color{blue} 4} \\
			{\color{green!70!black} 3 } & {\color{green!70!black} 1 } & {\color{green!70!black} 4 } & {\color{green!70!black} 2 } \\
		\end{pmatrix}.$$
It encodes the relation between the two orders for the points $p_i = (y_i,z_i)$ ($i=1,\ldots,4$): $\pi_\varphi(i) = j$ means that the $i$-th critical value $z_i=\varphi(y_i)$ is at $j$-th rank among critical values. 
\end{example}

%-----------------------------------------------
\subsection{Right-reduced functions, morsifications and their Morse rectangles}
\label{ssec:Morsifications}

We define now the notion of \emph{morsification} of a univariate singularity, paying attention to the intervals of definition of the associated Morse functions: 

\begin{definition}
\label{def:morsif}
Let $f\in \Rr\{x, y\}$ be such that $f(0, 0) = 0$ and $f(0, y) \not\equiv 0$.
Let $F_x(y) \in \Rr\{x, y\}$ be a primitive of $f$ in the sense of Formula (\ref{eq:PrimitiveOff}).
A \defi{Morse rectangle of $F_x(y)$} is a product $[0, \epsilon] \times I$, where $\epsilon > 0$ and $I$ is a compact interval neighborhood of the origin in the $y$-axis such that:
 \begin{enumerate}
 \item The primitive $F_x(y)$ is convergent on $[0, \epsilon] \times I$.
 \item $F_0 : I \to \Rr$ has $0$ as single critical point.
 \item $F_{x_0} : I \to \Rr$ is a Morse function for every $x_0 \in (0, \epsilon]$.
 \end{enumerate}

We say that $F_x(y)$ is a \defi{morsification} (of $F_0(y)$) if it admits a Morse rectangle.
\end{definition}

\begin{figure}[H]
\myfigure{0.8}{%
	\begin{tikzpicture}[scale=1.5]
\draw [->, >=latex, gray, thick](-0.5,0) -- (3,0) node[below,black] {$x$};
\draw [->, >=latex, gray, thick] (0,-2.0)--(0,2) node[left,black] {$y$};

\draw[brown, ultra thick, opacity=0.75] (0,-1.25) rectangle ++(1,2.5);

\draw[ultra thick, color=blue]  (0,0) .. controls (0.5,0) and (1,0) .. (2,1.1) node[right] {$\xi_i$};
\draw[ultra thick, color=blue]  (0,0) .. controls (0,0.25) and (1,1) .. (2.1,0.75) ;
\draw[ultra thick, color=blue]  (0,0) .. controls (0.25,0) and (1.25,0) .. (2,-0.75);
\draw[ultra thick, color=blue]  (0,0) .. controls (0.15,-0.05) and (1.5,0) .. (2,-1.75);
%\node [below, color=black] at (-0.4,0) {$O$};

\node[draw,circle, inner sep=1.5pt,color=black, fill=black] at (0,0){};

\draw[<->,>=latex,brown] (0,-1.5) -- ++(1,0) node[midway,below]{$\epsilon$};
\draw[<->,>=latex,brown] (-0.25,-1.25) -- ++(0,2.5) node[pos=0.25,left]{$I$};

\end{tikzpicture}%
}
\caption{A Morse rectangle of $f$.}
\label{fig:adaptedRectangle}
\end{figure}
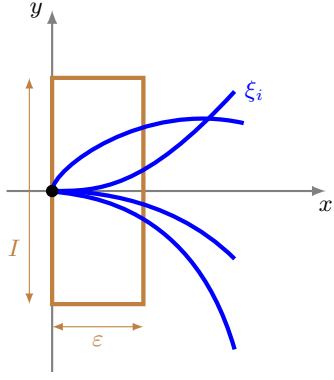

Let us introduce now a notion of reducedness of real series adapted to their study in the right half-plane $x \geq 0$. Geometrically, this means that we assume that the right semi-branches of $f$ are reduced in the divisor of $f$.
\begin{definition} 
\label{def:rightred}
The series $f\in \Rr\{x,y\}$ is \defi{right-reduced} if $f(0,0)=0$, $f(0,y) \not\equiv 0$ and if all the real roots $\xi_i$ of $f$ are pairwise distinct, that is, if the multi-set $\mathcal{R}_{\Rr}(f)$ is a set.
\end{definition}

\begin{example}
The series $(y^2 + x)^3(y^2- x)$ is right-reduced, but it is not reduced as an element of the ring $\Rr\{x, y\}$. 
\end{example}

\begin{proposition} 
\label{prop:morsif}
Let $f\in \Rr\{x, y\}$ be a right-reduced series and $F_x(y)$ be a primitive of $f$. 
Assume that the series $F_x(\xi_i) \in \Rr\{ x^{\frac{1}{\Nn}} \}$ are pairwise distinct when $\xi_i$ varies among the real Newton-Puiseux roots of $f$. 
Then a Morse rectangle $[0, \epsilon] \times I$ of $F_x(y)$ exists. Moreover, the bi-ordered critical graphs $(\Crit(F_{x_0}),\order_s,\order_t)$ of the Morse functions $F_{x_0} : I \to \Rr$ are isomorphic for all $x_0 \in (0, a]$.
\end{proposition}

\begin{proof}
We first choose a rectangle $[0,\epsilon] \times I$ included in the convergence disk of $f(x,y)$ and of $F_x(y)$. 
We may reduce $I$ in order that $F_0$ has a single critical point at $y =0$.
We may then diminish $\epsilon$ such that the roots $\xi_i : [0, \epsilon] \to I$ of $f$ converge on $[0,\epsilon]$ and that for every $x_0 \in (0, \epsilon]$, one has $\xi_i(x_0) \neq \xi_j(x_0)$ whenever $i \neq j$.

Fix $x_0 \in (0,\epsilon]$. Let us prove that $F_{x_0} : I \to \Rr$ has non-degenerate critical points.
The function $y \mapsto F_{x_0} (y)$ has a critical point at $y_0$ iff $f(x_0,y_0) = 0$.
Moreover this critical point is degenerate iff $\partial_y f(x_0,y_0) = 0$.
We notice that multiplication of $f$ by a unit does not change the nature of the critical point.
Indeed, let $f(x,y) = u(x,y) g(x,y)$ with $u(0,0) \neq 0$. We may assume that $u(x,y) \neq 0$ for $(x,y) \in (0,\epsilon] \times I$. Now $f(x_0,y_0) = 0$ iff $g(x_0,y_0) = 0$.
Moreover for such a critical point, using that $g(x_0,y_0)=0$:
$$
\partial_y f(x_0,y_0)= 0 
\iff \partial_y u(x_0,y_0) g(x_0,y_0) + u(x_0,y_0) \partial_y g(x_0,y_0) = 0
\iff \partial_y g(x_0,y_0) = 0.
$$
Hence one may assume that $f(x,y) = \prod_{i=1}^n (y-\xi_i) \cdot \prod_{l=1}^m [(y-\eta_l)(y-\overline{\eta_l})]$.
The function $y \mapsto F_{x_0}(y)$ having a non-degenerate critical point at $y_0 \in I$ is equivalent to the polynomial $f(x_0,y)$ having $y_0$ as a simple root.
As $f(x_0,y)$ factors into $\prod_{i=1}^n (y-\xi_i(x_0)) \cdot \prod_{l=1}^m [ (y-\eta_l(x_0))(y-\overline{\eta_l(x_0)}) ]$  and $\xi_i(x_0) \neq \xi_j(x_0)$ whenever $i \neq j$, we see that the roots of $f(x_0,y)$ on $I$ are exactly the real numbers $\xi_i(x_0)$. By the same condition, these numbers are pairwise distinct, which proves our claim.

We now prove that the critical values of $F_{x_0} : I \to \Rr$ are pairwise distinct.
As the series $F_x(\xi_i)$ are assumed to be pairwise distinct, we may diminish more $\epsilon >0$, such that for all $x_0 \in (0,\epsilon]$, one has $F_{x_0}(\xi_i(x_0)) \neq F_{x_0}(\xi_j(x_0))$ whenever $i\neq j$. This means that $F_{x_0} : I \to \Rr$ is a Morse function.

Finally, we prove that the bi-order remains constant for all $x_0 \in (0,\epsilon]$.
Fix $x_0 \in (0,\epsilon]$. Fix $i,j$ such that $\xi_i(x_0) < \xi_j(x_0)$. We may assume that 
$F(\xi_i(x_0)) < F(\xi_j(x_0))$ (the proof in the case of the opposite inequality is similar). 
Fix $x_1 \in (0,x_0)$. Then $\xi_i(x_1) < \xi_j(x_1)$ (otherwise, by the continuity of the function $\xi_j - \xi_i$,  there would exist $x_2 \in [x_1,x_0]$ such that $\xi_i(x_2) = \xi_j(x_2)$, which is impossible by our choices of $\epsilon$ and $I$). 
Similarly, we have $F(\xi_i(x_1)) < F(\xi_j(x_1))$.
In other words, the bi-ordered sets $(\Crit(F_{x_0}),\order_s,\order_t)$ and $(\Crit(F_{x_1}),\order_s,\order_t)$ are isomorphic.
\end{proof}

\begin{remark}
\sauteligne
\begin{enumerate}
 \item Notice that if $G_x(y)$ is another primitive of $f$, then $G_x(y) = g(x) + F_x(y)$. Therefore, for a fixed $x_0>0$, the graph of $G_{x_0} : I \to \Rr$ is a vertical translation of the graph of $F_{x_0} : I \to \Rr$, hence they have equivalent bi-ordered  critical graphs. 
 The bi-ordered critical graph is also independent of the choice of a Morse rectangle.

 \item Because the critical points of $F_{x_0}: I \to \Rr$ are $(\xi_i(x_0))_{1 \leq i \leq n}$ and the critical values are $(F(\xi_i(x_0)))_{1 \leq i \leq n}$, the ``right-reduced'' hypothesis implies that the critical points of $F_{x_0} : I \to \Rr$ are non-degenerate; the hypothesis on distinct $F_x(\xi_i)$ implies that the critical values of $F_{x_0}(y)$ are pairwise distinct when $x_0$ is small enough.
\end{enumerate}
\end{remark}

\begin{definition}
\label{def:combinmorsif}
 Let $F_x(y) \in \Rr\{x,y\}$ be a morsification. Its \defi{combinatorial type} is the isomorphism class of the bi-ordered critical graphs of the functions $F_{x_0} : I \to \Rr$ chosen as in Proposition \ref{prop:morsif}. 
\end{definition}

Our goal is to describe the combinatorial types of morsifications starting from the series $F_x(y) \in \Rr\{x,y\}$ defining them. This goal will be achieved in Theorem A, under the hypothesis that $f = \partial_y F_x$ satisfies the so-called \emph{injectivity condition}, explained in Subsection \ref{ssec:injCond}.

%%%%%%%%%%%%%%%%%%%%%%%%%%%%%%%%%%%%%%%%%%%%%%%%
\section{Considerations about trees}
\label{sec:contact-trees}

Since trees play a key role in our results, in this section we explain basic facts concerning them, partly following \cite[Section 1.4.1]{sorea2018shapes} and the references therein.

%-----------------------------------------------
\subsection{Abstract trees}
\label{ssec:abstrTrees}
 
A \defi{tree} is a topological space homeomorphic to a finite connected graph without cycles. Except when it is reduced to a point, a tree has an infinite number of points.  The \defi{valency} of a point of a tree is the number of connected components of $T \setminus \{P\}$. Its \defi{vertices} are its points of valency different from $2$ and its \defi{edges} are the closures of the connected components of the complement of its set of vertices. 
Given two points $P, Q$ of a tree, we denote by $[P, Q]$ the unique segment joining them.

For us, a \defi{rooted tree} has a marked point $O$ of valency $1$, called the \defi{root}. We choose this hypothesis about valency because all the rooted trees considered in this paper, namely the \emph{contact trees} of Subsection \ref{ssec:contree}, satisfy it.
Every rooted tree $T$ is endowed with a natural \defi{partial order} $\preceq_{T}$: given two distinct points $P$ and $Q$ of $T$, $P \preceq_{T} Q$ if and only if $[O,P]\subset [O,Q]$. A \defi{leaf} of $T$ is a maximal element for the partial order $\preceq_{T}$. 

Denote by $\mathcal{V}(T)$ the set of vertices, by $\mathcal{L}(T)$ the set of leaves and by 
$\mathcal{V}^\circ(T) = \mathcal{V}(T) \setminus (\mathcal{L}(T) \cup \{O\})$ the set of \defi{internal vertices} of $T$. 
If $P$ is an internal vertex of $T$, then an \defi{outgoing edge of $T$ at $P$} is by definition an edge $[P, Q]$ that is not contained in a segment of the form $[O,P]$. We denote by $\mathcal{E}_{T}^+(P)$ the set of outgoing edges of $T$ at $P$. These sets will be used in Definition \ref{def:plastr} for the formulation of the notion of \emph{planar tree}.

To any two points $P$ and $Q$ of $T$ we associate their \defi{greatest lower bound} $P \wedge Q$ relative to the partial order $\preceq_T$. That is (see Figure \ref{fig:wedgeOp}): 
$$[O,P\wedge Q] = [O,P] \cap [O,Q].$$

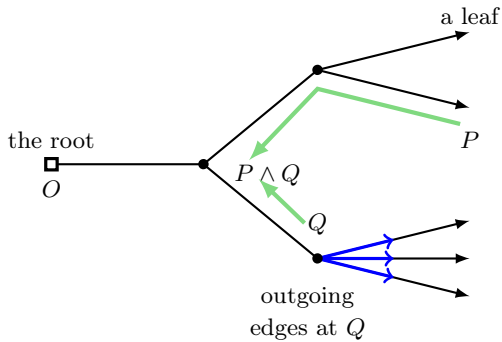
\begin{figure}[H]
\myfigure{0.5}{%
	\begin{tikzpicture}[scale=1]
% define points
\path
  (0,0) coordinate(R)
  ++(2,0) coordinate(S)
  ++(2,0) coordinate(T)
  (T)
  +(3,-2.5) coordinate(V1)
  +(3,2.5) coordinate(V2)
  (V1)
  +(4,-1) coordinate(W1)  
  +(4,0) coordinate(W2)  
  +(4,1) coordinate(W3)  
  (V2)
  +(4,-1) coordinate(W4)   
  +(4,1) coordinate(W5)  
;
% draw lines
\draw[thick]
  (R) -- (S)
  (S) -- (T)
  (T) -- (V1)
  (T) -- (V2)
;
\draw[->,>=latex,thick] (V1) -- (W1);
\draw[->,>=latex,thick]   (V1) -- (W2);
\draw[->,>=latex,thick]  (V1) -- (W3);
\draw[->,>=latex,thick]   (V2) -- (W4);
\draw[->,>=latex,thick]   (V2) -- (W5);

% outgoing edges
\draw[->, blue, very thick] (V1) -- ($(V1)!0.5!(W1)$) node[below left,black, text width=2cm,align=center]{outgoing edges at $Q$};
\draw[->,  blue, very thick] (V1) -- ($(V1)!0.5!(W2)$);
\draw[->,  blue, very thick] (V1) -- ($(V1)!0.5!(W3)$);

% draw points
\foreach \v/\t in {T,V1,V2}{
  \draw (\v) node[scale=4]{.};
}
\filldraw[very thick,fill=white] (R) ++ (-0.15,-0.15) rectangle ++(0.3,0.3);
% labels
\path
  (R) node[above=3pt]{the root}
  (R) node[below=3pt]{$O$}
  (W5) node[above]{a leaf}
  (W4) node[below=2mm]{$P$}
  (V1) node[above=2mm]{$Q$}
  (T) node[anchor = 170, inner sep=4mm]{$P \wedge Q$}
;

\draw[<-, >=latex, green!70!black!50, ultra thick] ($(T)+(-15:1.5)$) -- ($(V1)+(110:1)$);
\draw[<-, >=latex, green!70!black!50, ultra thick] ($(T)+(5:1.2)$) -- ($(V2)+(-90:0.5)$) -- ($(W4)+(-120:0.5)$);
\end{tikzpicture}%
}
\caption{A set of outgoing edges and the greatest lower bound of two vertices.}
\label{fig:wedgeOp}
\end{figure}

%-----------------------------------------------
\subsection{Planar trees}
\label{ssec:planarTrees} 

In this subsection we explain the notion of \emph{planar tree}, which is essential in the sequel, as one may associate canonically such a tree to any finite set of real Newton-Puiseux series (see Subsection \ref{ssec:contree}): 

\begin{definition} 
\label{def:plastr}
A \defi{planar structure} on a rooted tree $T$ is a choice of a total order $\order_P$ on each set $\mathcal{E}_{T}^+(P)$ of outgoing edges, when $P$ varies among the internal vertices of $T$. A \defi{planar tree} is a rooted tree endowed with a planar structure.
\end{definition}
 
The terminology \emph{planar structure} is motivated by the fact that such a structure is equivalent to the choice of an isotopy class of embeddings of the rooted tree in any given oriented plane. This equivalence would not be true any more if the root were of valency at least $2$. Indeed, in that case an isotopy class of embeddings in an oriented plane would only be fixed if one chooses moreover a cyclic order of the edges adjacent to the root.

When embedding canonically a planar tree $T$ in an oriented plane, one sees that its set of leaves $\mathcal{L}(T)$ is canonically totally ordered (see Figure \ref{fig:OrderLeaves}). This \defi{associated total order} may also be defined intrinsically (without mentioning an embedding into a plane) as follows: if $\ell_1, \ell_2$ are two distinct leaves of $T$ and $P := \ell_1 \wedge \ell_2$, then $\ell_1 < \ell_2$ if and only if $e_1 \order_P e_2$, where $e_1, e_2 \in \mathcal{E}^+_T(P)$ are the outgoing edges at $P$ going to $\ell_1$ and $\ell_2$ respectively.

\begin{figure}[H]
\myfigure{0.5}{%
	\begin{tikzpicture}[scale=1.5]
% define points
\path
  (0,0) coordinate(R)
  ++(2,0) coordinate(T)
  (T)
  +(2,-1.5) coordinate(V1)
  +(2,1.5) coordinate(V2)
  (V1)
  +(3,-1) coordinate(W1)  
  +(3,1) coordinate(W2)  
  (V2)
  +(3,-1) coordinate(W3)  
  +(3,0) coordinate(W4)   
  +(3,1) coordinate(W5)  
;
% draw lines
\draw[thick]
  (R) -- (T)
  (T) -- (V1)
  (T) -- (V2)
;
\draw[->,>=latex,thick] (V1) -- (W1);
\draw[->,>=latex,thick] (V1) -- (W2);
\draw[->,>=latex,thick] (V2) -- (W3);
\draw[->,>=latex,thick] (V2) -- (W4);
\draw[->,>=latex,thick] (V2) -- (W5);

% draw points
\filldraw[ultra thick,fill=white] (R) ++ (-0.1,-0.1) rectangle ++(0.2,0.2);
\foreach \v/\t in {T,V1,V2}{
  \draw (\v) node[scale=4]{.};
}
% arcs
\draw[->,>=latex,red!30, line width=2pt] ($(T) + (-50:1)$) arc (-50:50:1);
\draw[->,>=latex,red!30, line width=2pt] ($(V1) + (-30:1)$) arc (-30:30:1);
\draw[->,>=latex,red!30, line width=2pt] ($(V2) + (-30:1)$) arc (-30:30:1) node[pos=0,below,red,text width=1.5cm,align=center]{local ordering};

\draw[->,>=latex,blue!20, line width=3pt] ($(W1) + (0.5,-0.3)$) -- ($(W5) + (0.5,0.3)$) node[midway,right,text width=2cm,blue,align=center]{ordering of the leaves};

\end{tikzpicture}%
}
\caption{Canonical total order on the leaves of a planar tree.}
\label{fig:OrderLeaves}
\end{figure}
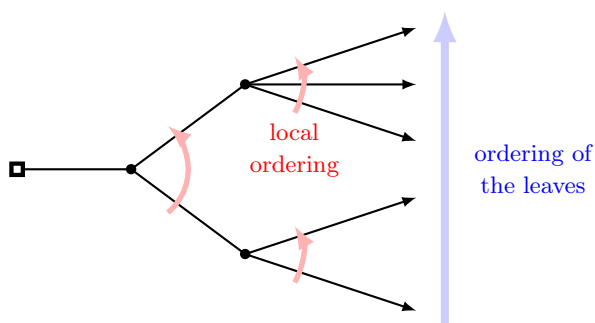

Not every total order on its set $\mathcal{L}(T)$ of leaves comes from a planar structure on a rooted tree $T$, as shown by the following proposition:
\begin{proposition}
\label{prop:compatord}
Let $\order$ be a total order on the set of leaves of a rooted tree $T$. 
The necessary and sufficient condition for $\order$ to come from a planar structure on $T$ is that for any two incomparable vertices $P, Q$ of $T$ (that is, vertices such that $P \npreceq_T Q$ and $Q \npreceq_T P$), the leaves $\preceq_T$-greater than $P$ are either all $\order$-smaller or all $\order$-bigger than the leaves $\preceq_T$-greater than $Q$. In this case, the total order $\order$ determines the planar structure uniquely.
\end{proposition}

\begin{proof}
    \emph{Let us assume first that $T$ is endowed with a planar structure.} Denote by $\order$ the associated total order on $\mathcal{L}(T)$. For each vertex $P$ of $T$, let $D(P)$ be the set of leaves $\preceq_T$-greater than $P$ (which may be thought as the set of \emph{descendants} of $P$, if $T$ is imagined as a genealogical tree). Consider two incomparable vertices $P, Q$ of $T$. Denote $R := P \wedge Q  \notin \{P, Q\}$. Let $e_P$ be the outgoing edge at $R$ directed towards $P$ and define similarly $e_Q$. We may assume, possibly after permuting $P$ and $Q$, that $e_P \order_R e_Q$. Choose $\ell_P \in D(P)$ and $\ell_Q \in D(Q)$. Then $e_P$ is the outgoing edge at $R$ directed towards $\ell_P$, and similarly $e_Q$ goes towards $\ell_Q$. The definition of the total order $\order$ on $\mathcal{L}(T)$ and the fact that $e_P \order_R e_Q$ imply that $\ell_P \order \ell_Q$. Therefore, all leaves in $D(P)$ are $\order$-smaller than all the leaves in $D(Q)$. 

    \emph{Let us assume now that $\mathcal{L}(T)$ is endowed with a total order $\order$ verifying the given condition.} Consider a vertex $P$ of $T$. We want to show that $\order$ determines a canonical total order $\order_P$ on $\mathcal{E}_{T}^+(P)$. If $P$ is a leaf, there is nothing to prove. Assume therefore that $P$ is not a leaf. Let $e_1$ and $e_2$ be two distinct outgoing edges at $P$. Let us write $e_i = [P, P_i]$. As the vertices $P_1$ and $P_2$ are incomparable, we know that the elements of $D(P_1)$ are either all $\order$-smaller or all $\order$-bigger than the elements of $D(P_2)$. In the first case we set $e_1 \order_P e_2$ and in the second one $e_2 \order_P e_1$. We get an antisymmetric binary relation $\order_P$ on the set $\mathcal{E}_{T}^+(P)$. As $\order$ is a total order, this is also the case for $\order_P$. 
\end{proof}

Proposition \ref{prop:compatord} motivates the following definition, which will be used in the formulation of Proposition \ref{prop:realtotplan}:
\begin{definition} 
 \label{def:plantotord}
 Let $T$ be a rooted tree. A total order on the set $\mathcal{L}(T)$ of leaves of $T$ is called \defi{planar relative to $T$} if it is determined by a planar structure on $T$. 
\end{definition}

\begin{example}
\label{ex:planarRelativeToT}
Consider the abstract rooted tree $T$ of Figure \ref{fig:exPlanarRelativeToT}. Take the following total order on $\mathcal{L}(T)= \{\ell_1, \ell_2, \ell_3\}$:
$$\ell_2 < \ell_1 < \ell_3. $$
The vertices $P := \ell_1$ and $Q := \ell_2 \wedge \ell_3$ are incomparable, but $\ell_1$, which is the only leaf $\preceq_T$-greater than $P$ is neither $\order$-smaller nor $\order$-bigger than both $\ell_2$ and $\ell_3$, which are the leaves $\preceq_T$-greater than $Q$. Therefore this total order is not planar relative to $T$. This example shows also that it is important to allow the vertices $P$ and $Q$ appearing in Proposition \ref{prop:compatord} to be leaves. 

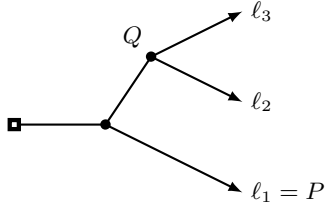
\begin{figure}[H]
\small
\myfigure{0.6}{
	\begin{tikzpicture}[scale=1]
% define points
\path
  (0,0) coordinate(R)
  ++(2,0) coordinate(T)
  (T)
  +(3,-1.5) coordinate(V1)
  +(1,1.5) coordinate(V2)
  (V2)
  +(2,-1) coordinate(W1)  
  +(2,0) coordinate(W2)    
  +(2,1) coordinate(W3)  
;

\draw[thick]
  (R) -- (T)
  %(T) -- (V1)
  (T) -- (V2)
;
\draw[->,>=latex,thick] (T) -- (V1);
\draw[->,>=latex,thick] (V2) -- (W1);
%\draw[->,>=latex,thick] (V2) -- (W2);
\draw[->,>=latex,thick] (V2) -- (W3);

\filldraw[ultra thick,fill=white] (R) ++ (-0.1,-0.1) rectangle ++(0.2,0.2);

%\path (T) node[above left]{$[x]$};

\path (V2) node[above left]{$Q$};
\foreach \v in {T, V2}{
  \path  (\v) node[scale=4]{.};
}
\path (V1) node[right]{$\ell_1=P$};
\path (W1) node[right]{$\ell_2$};
%\path (W2) node[right]{$\gamma_3$};
\path (W3) node[right]{$\ell_3$};

%\path (T)--(V1) node[below,pos=0.3]{$- x$};

%\path (T)--(V2) node[above,pos=0.3]{$x$};

%\path (V2)--(W1) node[below,pos=0.3]{$0$};

%\path (V2)--(W2) node[ above, pos=0.5]{$x^3$};

%\path (V2)--(W3) node[above,pos=0.3]{$2x^3$};
\end{tikzpicture}%

}
\caption{The abstract rooted tree from Example \ref{ex:planarRelativeToT}.}
\label{fig:exPlanarRelativeToT}
\end{figure}

\end{example}

%-----------------------------------------------
\subsection{The wedge map of a planar tree}
\label{ssec:wedge map} 

Let $(X,\order)$ be a finite totally ordered set. Denote its elements by $x_1 < x_2 < \cdots < x_n$.
A \defi{basic interval} of $X$ is a subset $\{ x_i,x_{i+1} \}$ of two successive elements of $X$.
Denote by $BI(X)$ the set of basic intervals of $(X,\le)$. This set is empty if and only if $n \le 1$. 

%For more details the reader is referred to \cite[page 5]{kitaev}, \cite[page 13]{ghys_promenade}.

Let $T$ be a planar tree. The sets $\mathcal{L}(T)$ and $\mathcal{E}^+_T(P)$, where $P$ is an internal vertex of $T$, are therefore naturally totally ordered, as explained in Subsection \ref{ssec:planarTrees}. 
Let  $\{ \ell, \ell' \}$ be a basic interval of $\mathcal{L}(T)$.  Denote $P := \ell \wedge \ell'$. Let $e$ and $e'$ be the outgoing edges going from $P$ to the leaves $\ell$ and $\ell'$ (see Figure \ref{fig:basicInt}). Then $\{ e, e' \}$ is a basic interval of $\mathcal{E}^+_T(P)$, by the definition of the total order on the set of leaves of a planar tree. This construction defines the \defi{wedge map} of the planar tree $T$:
$$W : 
BI\big(\mathcal{L}(T)\big) \longrightarrow \bigsqcup_{P\in\mathcal{V}^\circ (T)} BI\big(\mathcal{E}^+_T(P)\big).$$

Note that when $P \in \mathcal{V}^\circ (T)$, one has $BI\big(\mathcal{E}^+_T(P)\big) \neq \emptyset$.

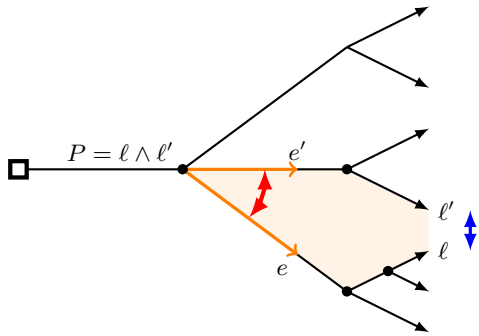
\begin{figure}[H]
\myfigure{0.9}{%
	\begin{tikzpicture}[scale=1.2]
% define points
\path
  (0,0) coordinate(R)
  ++(2,0) coordinate(T)
  (T)
  +(2,-1.5) coordinate(V1)
  +(2,0) coordinate(V2)
  +(2,1.5) coordinate(V3)
  (V1)
  +(0.5,0.25) coordinate(VV1)  
  +(1,-0.5) coordinate(W1)    
  +(1,0.5) coordinate(W3)  
  (VV1)
  +(0.5,-0.25) coordinate(W2)  
  (V2)
  +(1,-0.5) coordinate(W4)  
  +(1,0.5) coordinate(W5)   
  (V3)
  +(1,-0.5) coordinate(W6)  
  +(1,0.5) coordinate(W7)   
;

\fill[orange!10] (T) -- (V1) -- (W3) -- (W4) -- (V2) -- cycle;

% draw lines
\draw[thick]
  (R) -- (T)
  (T) -- (V1)
  (T) -- (V2)
  (T) -- (V3)
;
\draw[->,>=latex,thick] (V1) -- (W1);
\draw[->,>=latex,thick] (VV1) -- (W2);
\draw[->,>=latex,thick] (V1) -- (W3);
\draw[->,>=latex,thick] (V2) -- (W4);
\draw[->,>=latex,thick] (V2) -- (W5);
\draw[->,>=latex,thick] (V3) -- (W6);
\draw[->,>=latex,thick] (V3) -- (W7);

\draw[->, orange, very thick] (T) -- ($(T)!0.7!(V1)$) node[below left,black]{$e$};
%\draw[orange, very thick, shorten >=2pt] (T) -- (V1) -- (W3);
\draw[->,orange, very thick] (T) -- ($(T)!0.7!(V2)$) node[above,black]{$e'$};
%\draw[orange, very thick, shorten >=2pt] (T) -- (V2) -- (W4);

% draw points
\filldraw[ultra thick,fill=white] (R) ++ (-0.1,-0.1) rectangle ++(0.2,0.2);
\foreach \v/\t in {T,V1,V2,VV1}{
  \draw (\v) node[scale=4]{.};
}
% intervals
\draw[<->,>=latex,red, ultra thick, shorten <=2pt] ($(T) + (-40:1)$) arc (-40:0:1) node[midway,right,red,text width=1.5cm,align=center]{};
\draw[<->,>=latex,blue, very thick] ($(W3) + (0.5,-0.0)$) -- ($(W4) + (0.5,0.)$) node[midway,right,text width=2cm,blue]{};
% labels
\path (T) node[above left]{$P= \ell \wedge \ell'$};
\path (W3) node[right]{$\ell$};
\path (W4) node[right]{$\ell'$};
\end{tikzpicture}%
}
\caption{A basic interval of $\mathcal{L}(T)$ and its image in $\mathcal{E}^+_T(P)$ by the wedge map}.
\label{fig:basicInt}
\end{figure}

The following proposition will be crucial in Subsection \ref{ssec:signsdiff}, as well as in Subsection \ref{ssec:injCond}, in order to define the \emph{injectivity condition}: 

\begin{proposition}
\label{prop:wedgemap}
The wedge map $W$ of a planar tree is bijective.
\end{proposition}

\begin{proof}
The source and target of the wedge map $W$ have the same number of elements, as may be proved by induction on the number of leaves of $T$.  Therefore, in order to prove that $W$ is bijective, it is enough to prove that it is surjective. Consider a vertex $P \in \mathcal{V}^\circ (T)$ and a basic interval $\{e_j , e_{j+1} \}$ of $\mathcal{E}^+_T(P)$, with $e_j <_P e_{j+1}$. Let $\ell_{\iota(j)}$ be the $\order$-biggest leaf among the descendants of $e_j$ and $\ell_{\iota(j+1)}$ be the $\order$-lowest leaf among the descendants of $e_{j+1}$. By the construction of the total order $\order$ on $\mathcal{L}(T)$ explained in Subsection \ref{ssec:planarTrees}, we have $\ell_{\iota(j)} < \ell_{\iota(j+1)}$. By Proposition \ref{prop:compatord}, we see that $\{ \ell_{\iota(j)},  \ell_{\iota(j+1)}\}$ is a basic interval of $(\mathcal{L}(T), \order)$. As results from the definition of the wedge map, its image by $W$ is the basic interval $\{e_j , e_{j+1} \}$ of $\mathcal{E}^+_T(P)$. This shows that $W$ is surjective, therefore bijective.   
\end{proof}

Proposition \ref{prop:wedgemap} generalizes \cite[Corollary 2.20]{sorea_portugaliae}, which concerned only the case where the rooted tree was \emph{binary}, that is, where all its vertices had valency $1$ or $3$.

%-----------------------------------------------
\subsection{Contact trees}
\label{ssec:contree}
 
Let us consider a finite set of Newton-Puiseux series 
$\mathcal{N} = \{\gamma_1,\ldots,\gamma_n \} \subset \Kk\{x^{\frac{1}{\Nn}}\}$, such that $\gamma_i(0)=0$ for all $i=1,\ldots,n$.
The \defi{contact tree} of the set $\mathcal{N}$, denoted by $T_\Kk(\mathcal{N})$ or by $T_\Kk(\gamma_1, \dots, \gamma_k)$, is a rooted tree encoding the valuations of pairwise differences of the elements of $\mathcal{N}$. It is canonically determined by the ultrametric distance $\mathrm{d}:\mathcal{N}\times \mathcal{N} \rightarrow (0,+\infty)$ defined by: 
\begin{equation*}
\label{eq:d}
 \mathrm{d}(\gamma_i,\gamma_j) :=\frac{1}{\val (\gamma_j-\gamma_i)}, 
\end{equation*} 
whenever $\gamma_i \neq \gamma_j$.
For details, we refer the reader to \cite[Section 9.4]{ultrametricsLipschitz} and references therein.
The contact tree is a version of the so-called \emph{Kuo-Lu tree}, introduced in \cite{KuoLu} (see \cite[Section 1.6.6]{handbookCurves}).

\medskip

Let us explain informally how to construct $T_\Kk(\mathcal{N})$ by gluing compact segments identified to $[0, \infty]$, one segment per series. Associate a copy $I_i$ of the interval $[0, \infty]$ to each series $\gamma_i \in \mathcal{N}$. The point of $I_i$ whose coordinate is $a \in [0, \infty]$ represents the formal monomial $x^a$. If $\gamma_i, \gamma_j \in \mathcal{N}$ are such that $\gamma_i \neq \gamma_j$, then glue the segments $[0, \val(\gamma_j - \gamma_i)]$ of the intervals $I_i$ and 
$I_j$ by identifying the points having the same coordinate in $[0, \infty]$. This gluing process leads to a tree which is by definition the contact tree $T_\Kk(\mathcal{N})$. All the points of coordinate $0$ of the intervals $I_i$ get identified to a point $O \in T_\Kk(\mathcal{N})$, which is chosen as the root. As $\val(\gamma_j - \gamma_i) > 0$ whenever $i \neq j$, the root is of valency $1$. The set $\mathcal{L}(T_\Kk(\mathcal{N}))$ of leaves of $T_\Kk(\mathcal{N})$ is in canonical bijection with the set $\mathcal{N}$. We will identify them using this bijection:
 $$\mathcal{N} = \mathcal{L}(T_\Kk(\mathcal{N})).$$

\begin{example}
\label{ex:contactTree}
Consider the set $\mathcal{N}$ consisting of the following real Newton-Puiseux series: 
$\gamma_1 = -x$, $\gamma_2=x$, $\gamma_3=x+x^3$ and $\gamma_4=x+2x^3$. The corresponding intervals $I_i$ are drawn on the left of Figure \ref{fig:exContactTree}, the marked points being those whose coordinates are exponents of monomials appearing in $\gamma_i$.
The contact tree $T_\Kk(\mathcal{N})$ is drawn on the right. The monomial $x^r$ corresponding to a vertex is written between brackets as a decoration. The corresponding term $c_i x^r$ in each series $\gamma_i$ is written as a decoration of the edge going towards $\gamma_i$,  seen as a leaf of $T_\Kk(\mathcal{N})$. Note that in this example all Newton-Puiseux series have integral exponents. One could restrict to this situation throughout the paper by making a change of variable of the form $x = x_1^N$, for a value $N \in \Nn^*$ divisible by the denominators of all the exponents appearing in the complex Newton-Puiseux roots of $f(x,y)$.

\begin{figure}[H]
\small
\myfigure{1}{
	\begin{tikzpicture}[scale=0.8]

\begin{scope}[yshift=0cm]
\draw[thick] (0,0) -- (2,0)  -- (4,0);
\draw[->,>=latex,thick] (4,0) -- ++(2,0) node[right]{$\gamma_4 = x + 2x^3$};
\filldraw[ultra thick,fill=white] (0,0) ++ (-0.1,-0.1) rectangle ++(0.2,0.2);
\node [above] at (2,0){$[x]$};
\node[below, scale=0.8] at (2.4,0) {$x$};
\node [above] at (4,0){$[x^3]$};
\node[below, scale=0.8] at (4.4,0) {$2x^3$};
\path  (2,0) node[scale=4]{.};
\path  (4,0) node[scale=4]{.};
\end{scope}

\begin{scope}[yshift=-1.5cm]
\draw[thick] (0,0) -- (2,0)  -- (4,0);
\draw[->,>=latex,thick] (4,0) -- ++(2,0) node[right]{$\gamma_3 = x + x^3$};
\filldraw[ultra thick,fill=white] (0,0) ++ (-0.1,-0.1) rectangle ++(0.2,0.2);
\node [above] at (2,0){$[x]$};
\node[below, scale=0.8] at (2.4,0) {$x$};
\node [above] at (4,0){$[x^3]$};
\node[below, scale=0.8] at (4.4,0) {$x^3$};
\path  (2,0) node[scale=4]{.};
\path  (4,0) node[scale=4]{.};
\end{scope}

\begin{scope}[yshift=-3cm]
\draw[thick] (0,0) -- (2,0) ;
\draw[->,>=latex,thick] (2,0) -- ++(4,0) node[right]{$\gamma_2 = x$};
\filldraw[ultra thick,fill=white] (0,0) ++ (-0.1,-0.1) rectangle ++(0.2,0.2);
\node [above] at (2,0){$[x]$};
\node[below, scale=0.8] at (2.4,0) {$x$};

\path  (2,0) node[scale=4]{.};
\end{scope}

\begin{scope}[yshift=-4.5cm]
\draw[thick] (0,0) -- (2,0) ;
\draw[->,>=latex,thick] (2,0) -- ++(4,0) node[right]{$\gamma_1 = -x$};
\filldraw[ultra thick,fill=white] (0,0) ++ (-0.1,-0.1) rectangle ++(0.2,0.2);
\node [above] at (2,0){$[x]$};
\node[below, scale=0.8] at (2.4,0) {$- x$};

\path  (2,0) node[scale=4]{.};
\end{scope}
\end{tikzpicture}%
\qquad\qquad
	\begin{tikzpicture}[scale=1]
% define points
\path
  (0,0) coordinate(R)
  ++(2,0) coordinate(T)
  (T)
  +(3,-1.5) coordinate(V1)
  +(1,1.5) coordinate(V2)
  (V2)
  +(2,-1) coordinate(W1)  
  +(2,0) coordinate(W2)    
  +(2,1) coordinate(W3)  
;

\draw[thick]
  (R) -- (T)
  %(T) -- (V1)
  (T) -- (V2)
;
\draw[->,>=latex,thick] (T) -- (V1);
\draw[->,>=latex,thick] (V2) -- (W1);
\draw[->,>=latex,thick] (V2) -- (W2);
\draw[->,>=latex,thick] (V2) -- (W3);

\filldraw[ultra thick,fill=white] (R) ++ (-0.1,-0.1) rectangle ++(0.2,0.2);

\path (T) node[above left]{$[x]$};

\path (V2) node[above left]{$[x^3]$};
\foreach \v in {T, V2}{
  \path  (\v) node[scale=4]{.};
}
\path (V1) node[right]{$\gamma_1$};
\path (W1) node[right]{$\gamma_2$};
\path (W2) node[right]{$\gamma_3$};
\path (W3) node[right]{$\gamma_4$};

\path (T)--(V1) node[below,pos=0.3]{$- x$};

\path (T)--(V2) node[above,pos=0.3]{$x$};

\path (V2)--(W1) node[below,pos=0.3]{$0$};

\path (V2)--(W2) node[ above, pos=0.5]{$x^3$};

\path (V2)--(W3) node[above,pos=0.3]{$2x^3$};
\end{tikzpicture}%

}
\caption{Construction of a contact tree.}
\label{fig:exContactTree}
\end{figure}
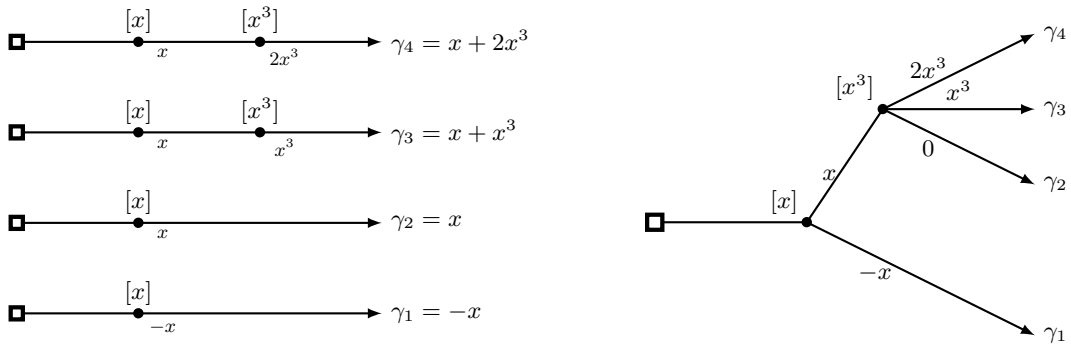
\end{example}

\medskip

Suppose now that the finite set $\mathcal{N}$ consists only of \emph{real} Newton-Puiseux series $\xi_i \in \Rr\{x^{\frac{1}{\Nn}}\}$. It acquires then a canonical total order, by restriction of the \emph{real total order} $\order_{\Rr}$ of Definition \ref{def:realtot}. Therefore, we also call it the \defi{real total order} on $\mathcal{N}$. It is planar relative to the rooted tree $T_\Rr(\mathcal{N})$:
\begin{proposition}
 \label{prop:realtotplan}
 Let $\mathcal{N}$ be a finite subset of $\Rr\{x^{\frac{1}{\Nn}}\}$.
 The real total order on $\mathcal{N}= \mathcal{L}(T_\Rr(\mathcal{N}))$ is planar relative to the tree $T_\Rr(\mathcal{N})$, in the sense of Definition \ref{def:plantotord}. 
\end{proposition}

\begin{proof}
Let $\xi_i$ and $\xi_j$ be two distinct elements of $\mathcal{N}$. By definition of the real total order, $\xi_i \order_{\Rr} \xi_j$ if and only if $\lc(\xi_j-\xi_i) > 0$ (see Subsection \ref{ssec:NPs}). Denote $k := \val(\xi_j-\xi_i) > 0$. Then $\xi_i = \xi + a_k x^k + hot$ and $\xi_j = \xi + b_k x^k + hot$, where $a_k \neq b_k$ and $\xi$ is a Newton-Puiseux polynomial of degree $< k$.   
Therefore $\lc(\xi_j-\xi_i) = b_k-a_k$, which shows that $\xi_i \order_{\Rr} \xi_j$ if and only if $a_k < b_k$. This implies easily the condition for planarity described in Proposition \ref{prop:compatord} (for more details, see \cite[Section 1.7.2]{sorea2018shapes}).
\end{proof}

A different version of real contact tree of a finite set of real Newton-Puiseux series was introduced in \cite[Section 6.3]{koike-parusinski} by Koike and Parusinski.

%-----------------------------------------------
\subsection{The real and complex contact trees of right-reduced series}
\label{ssec:contreerrseries}
 
Let $f \in \Rr\{x,y\}$ be a right-reduced series. 
We will distinguish the real Newton-Puiseux roots of $f$ from the non-real ones. 
Therefore, relation (\ref{eq:WeierstrassPrepThm}) becomes:
\begin{equation}
\label{eq:xiAndEta}
     f(x,y) = u(x,y)\cdot \prod_{i=1}^n (y-\xi_i) \cdot\prod_{l=1}^m \big[(y-\eta_l)(y-\overline{\eta_l})\big],
\end{equation}
where:
\begin{itemize}
 \item $\xi_i \in \mathcal{R}_{\Rr}(f)$, $i=1,\ldots,n$, are the real roots of $f$,
which are pairwise distinct by the hypothesis that $f$ is right-reduced;
 \item $\eta_l, \overline{\eta_l} \in \mathcal{R}_{\Cc}(f)\setminus \mathcal{R}_{\Rr}(f)$, $l=1,\ldots,m$, are the non-real roots of $f$; they are not necessarily pairwise distinct.
\end{itemize} 
Recall that $u$ is a unit (i.e.{} $u(0,0) \neq 0$) and since $f\in \Rr\{x,y\}$ we have that $u\in \Rr\{x,y\}$.

Denote by $T_{\Cc}(f)$ the contact tree of the set of all the Newton-Puiseux roots (real or complex) of $f$. 
Similarly, denote by $T_{\Rr}(f)$ the contact tree of the set of real Newton-Puiseux roots of $f$, namely $T_{\Rr}(\xi_1,\ldots,\xi_n)$. 
Note that \emph{$T_{\Rr}(f)$ is a rooted sub-tree of $T_{\Cc}(f)$} and that $T_{\Rr}(f)$ is canonically planar, by choosing the real total order on its set of leaves (see Proposition \ref{prop:realtotplan}). We say that this planar structure is the \defi{real planar structure} of $T_{\Rr}(f)$. We will define in Section \ref{sec:theorem} below a second planar structure on it, the \emph{integrated planar structure}.

%%%%%%%%%%%%%%%%%%%%%%%%%%%%%%%%%%%%%%%%%%%%%%%%
\section{Area series and their valuations}
\label{sec:injectivity}

Throughout this section, we assume that $f \in \Rr\{x,y\}$ is a right-reduced series and that its real Newton-Puiseux roots $\xi_i$ satisfy: $ \xi_1 \order_{\Rr} \xi_2 \order_{\Rr} \cdots \order_{\Rr} \xi_n$. 
Let $F_x(y) \in \Rr\{x,y\}$ be a primitive of $f$. Its \emph{area series} are the successive differences $F_x(\xi_{i+1}) - F_x(\xi_i)$. We compute their valuations and we prove that they may be described using a strictly increasing function on the real contact tree $T_{\Rr}(f)$, the \emph{integrated exponent function}. Then we deduce the valuations of the differences $F_x(\xi_j) - F_x(\xi_i)$ for $j - i \geq 2$, whenever a non-vanishing hypothesis is satisfied.

%-----------------------------------------------
\subsection{Area series}
\label{ssec:areas}

Consider a Morse rectangle $[0, \epsilon] \times I$ of $F_x(y)$ (see Definition \ref{def:morsif}). The numerical series $F_{x_0}(\xi_i(x_0))$ converges for every $x_0 \in [0, \epsilon]$ and its sum is a critical value of the function $F_{x_0} : I \to \Rr$. For this reason, we say that $F_x(\xi_i)$ is a \defi{critical value series}. 

In order to compare two critical values 
$F_{x_0}(\xi_i(x_0))$ and $F_{x_0}(\xi_j(x_0))$ of $F_{x_0}$ when $x_0 \in (0, \epsilon]$, we will first evaluate the initial terms of the differences 
\begin{equation}
 \label{eq:areaseries}
 S_i := F_x(\xi_{i+1})-F_x(\xi_i) \ \in \Rr\{x^\frac{1}{\Nn}\}
\end{equation} 
of consecutive critical value series. For all $i=1,\ldots,n-1$, we have:
\begin{equation*}
\label{eq:Sl}
S_i = \int_{\xi_i}^{\xi_{i+1}} f(x,t) \, \dd t 
\end{equation*}
by the definition of $F_x(y)$. 
Therefore, $S_i(x_0)$ is the signed area of the region contained between the interval $[\xi_i(x_0), \xi_{i+1}(x_0)]$ of the $y$-axis and the graph of $y \mapsto f(x_0,y)$ (see Figure \ref{fig:areas}). For this reason, we say that $S_i$ is the \defi{$i$-th area series of $f$}. Notice that the signs of the areas $S_i(x_0)$ alternate when $x_0 \in (0, \epsilon]$, since the hypothesis of right-reducedness of $f$ implies that $y \mapsto f(x_0,y)$ has only simple roots $\xi_i(x_0)$. Hence the critical points of $y \mapsto F_{x_0}(y)$ alternate between local minima and local maxima.

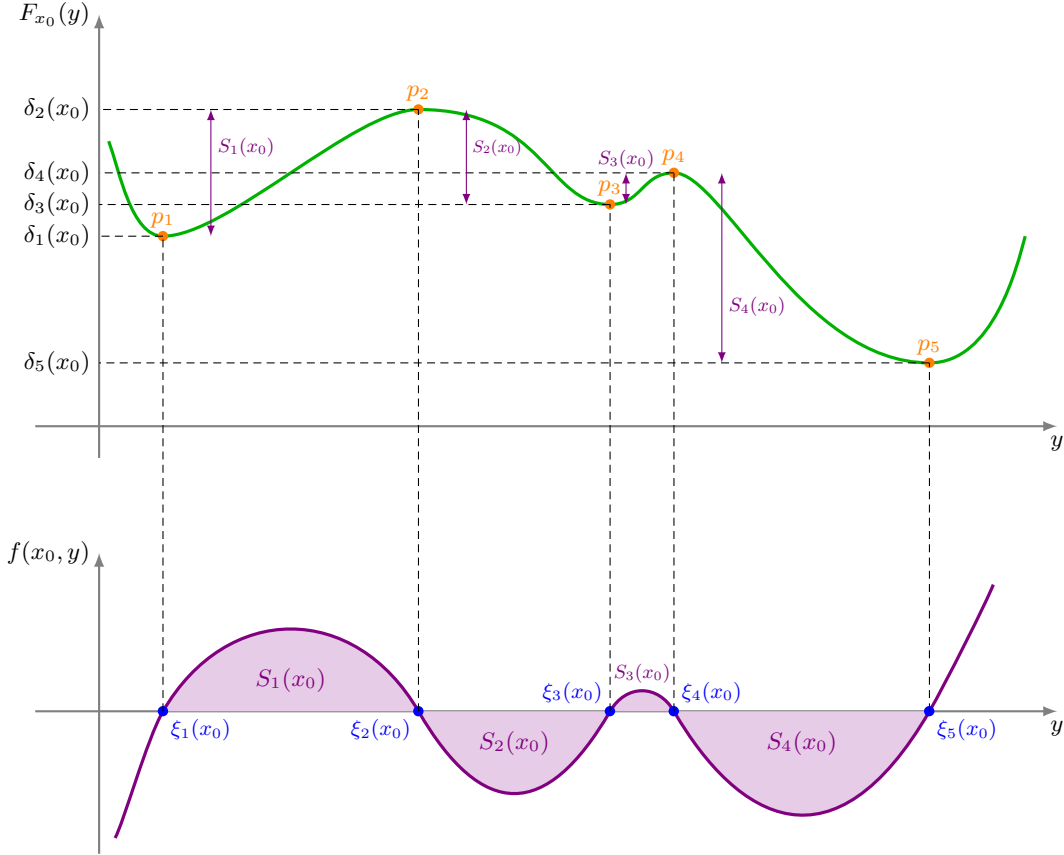
\begin{figure}[H]
\myfigure{0.7}{%
	\begin{tikzpicture}[scale=1.2]

% Low picture : f_{x_0}(y)
\begin{scope}
	% Axes
	\draw[->,>=latex,thick, gray] (-1,0)--(15,0) node[below,black] {$y$};
	\draw[->,>=latex,thick, gray] (0,-2.25)--(0,2.5) node[left,black] {$f(x_0,y)$};

  \coordinate (p1) at (1,0);
  \coordinate (p2) at (5,0);
  \coordinate (p3) at (8,0);
  \coordinate (p4) at (9,0);
  \coordinate (p5) at (13,0);

	%\draw[very thick, red]  plot [smooth, tension=0.5] coordinates {(p0) (p1) (p2) (p3) (p4) (p5) (p6)};

\def\mygraph{
(0.25,-2) ..controls +(60:0.25) and +(-120:0.5)..  (p1) 
    ..controls +(60:2) and +(120:2).. (p2) 
     ..controls +(-60:2) and +(-120:2)..  (p3) 
       ..controls +(60:0.5) and +(120:0.5)..  (p4) 
       ..controls +(-60:2) and +(-120:3)..  (p5) 
      ..controls +(60:0.5) and +(-110:0.25).. (14,2)
}

\begin{scope}
  \clip (1,-4) rectangle (13,4);
  \fill[fill=violet!20]  (0,0) -- \mygraph -- (14,0)  -- cycle;
\end{scope}

\draw[very thick, violet]  \mygraph;

\foreach \i in {1,...,5}{
  \draw (p\i) node[scale=4, blue]{.};
  %\node[below, blue, scale=0.9] at (p\i) {$\xi_\i(x_0)$};
}
\node[below right, blue, scale=0.9] at (p1) {$\xi_1(x_0)$};
\node[below left, blue, scale=0.9] at (p2) {$\xi_2(x_0)$};
\node[above left, blue, scale=0.9] at (p3) {$\xi_3(x_0)$};
\node[above right, blue, scale=0.9] at (p4) {$\xi_4(x_0)$};
\node[below right, blue, scale=0.9] at (p5) {$\xi_5(x_0)$};

\node[violet] at (3,0.5) {$S_1(x_0)$};
\node[violet] at (6.5,-0.5) {$S_2(x_0)$};
\node[violet,scale=0.8] at (8.5,0.6) {$S_3(x_0)$};
\node[violet] at (11,-0.5) {$S_4(x_0)$};

\end{scope}

% Top picture : F_{x_0}(y)
\begin{scope}[yshift=4.5cm]
	% Axes
	\draw[->,>=latex,thick, gray] (-1,0)--(15,0) node[below,black] {$y$};
	\draw[->,>=latex,thick, gray] (0,-0.5)--(0,6.5) node[left,black] {$F_{x_0}(y)$};

  \coordinate (q1) at (1,3);
  \coordinate (q2) at (5,5);
  \coordinate (q3) at (8,3.5);
  \coordinate (q4) at (9,4);
  \coordinate (q5) at (13,1);

	%\draw[very thick, red]  plot [smooth, tension=0.5] coordinates {(p0) (p1) (p2) (p3) (p4) (p5) (p6)};

\def\mygraph{
(0.15,4.5) ..controls +(-70:0.5) and +(180:0.5)..  (q1) 
    ..controls +(0:1) and +(180:1).. (q2) 
     ..controls +(0:2) and +(180:1)..  (q3) 
       ..controls +(0:0.5) and +(180:0.5)..  (q4) 
       ..controls +(0:0.75) and +(180:2)..  (q5) 
      ..controls +(0:1) and +(-105:0.5).. (14.5,3)
}

\draw[very thick, green!70!black]  \mygraph;

\foreach \i in {1,...,5}{
  \draw (q\i) node[scale=4, orange]{.};
  \node[above, orange] at (q\i) {$p_\i$};
}

\foreach \i in {1,...,5}{
  \draw[densely dashed] (q\i) -- (q\i -| 0,0) node[left] {$\delta_{\i}(x_0)$};
}

\draw[<->,>=latex,violet] (1.75,3) -- ++(0,2) node[pos=0.7,right,scale=0.8]{$S_1(x_0)$};
\draw[<->,>=latex,violet] (5.75,5) -- ++(0,-1.5) node[pos=0.4,right,scale=0.7]{$S_2(x_0)$};
\draw[<->,>=latex,violet] (8.25,4) -- ++(0,-0.5) node[pos=0.1,above,scale=0.8]{$S_3(x_0)$};
\draw[<->,>=latex,violet] (9.75,1) -- ++(0,3) node[pos=0.3,right,scale=0.8]{$S_4(x_0)$};

\end{scope}

\foreach \i in {1,...,5}{
  \draw[densely dashed] (p\i) -- (q\i);
}

\end{tikzpicture}%
}
\caption{The graphs of $y \mapsto f(x_0,y)$ (below) and its primitive $y \mapsto F_{x_0}(y)$ (above).}
\label{fig:areas}
\end{figure}

In order to compare two critical values $F_{x_0}(\xi_i(x_0))$ and $F_{x_0}(\xi_j(x_0))$ when $j - i \geq 2$, we need to determine the sign of the initial coefficient of the difference:
\begin{equation}
\label{eq:sumSi}
F_x(\xi_j) - F_x(\xi_i) = S_i + S_{i+1} + \cdots + S_{j-1}. 
\end{equation}

Denote by $s_r x^{\sigma_r}$ the initial term of the area series $S_r$:
\begin{equation} 
\label{eq:initarea} 
     S_r = s_r x^{\sigma_r} + \hot.
\end{equation}

%-----------------------------------------------
\subsection{The valuations of the area series}
\label{ssec:valarea}

Let $P$ be a point of $T_\Rr(f)$ or $T_\Cc(f)$. Its \defi{exponent} $E(P)$ is the exponent $k \in [0, \infty]$ of the monomial $x^k$ attached to $P$. In particular, if $P = \gamma_i \wedge \gamma_j$, where $\gamma_i, \gamma_j \in \mathcal{R}_{\Cc}(f)$ (see formula (\ref{eq:WeierstrassPrepThm})), then $E(P) = \val(\gamma_j-\gamma_i)$.
The following proposition generalizes \cite[Proposition 3.1]{sorea_portugaliae}, which concerned the case when $f$ had only real Newton-Puiseux roots with integer exponents, and that the associated real contact tree was binary.  It shows that the valuations of the area series $S_l$ may be computed combinatorially on the real contact tree $T_{\Rr}(f)$, using the embedding $T_{\Rr}(f) \subset T_{\Cc}(f)$:
\begin{proposition}
\label{prop:sigmai}
    Let $r \in \{1, \dots, n-1 \}$ and $S_r = F_x(\xi_{r+1}) - F_x(\xi_r) = s_r x^{\sigma_r} + \hot$, with $s_r \in \Rr^*$. Denote $P := \xi_r \wedge \xi_{r+1}$. Then: 
$$\sigma_r = E(P) + \sum_{\gamma\in \mathcal{R}_{\Cc}(f)} E(P \wedge \gamma).$$
\end{proposition}

\begin{proof}
By formula (\ref{eq:xiAndEta}) 
we have:  
$$S_r 
= F_x(\xi_{r+1}) - F_x(\xi_{r}) 
= \int_{\xi_{r}}^{\xi_{r+1}} f(x,t) \, \dd t 
= \int_{\xi_{r}}^{\xi_{r+1}} u(x,t)\cdot \prod_{i=1}^n (t-\xi_i) \cdot \prod_{l=1}^m |t-\eta_l|^2 \, \dd t. 
$$
Let us make the change of variables $t = \xi_r + \tau(\xi_{r+1}-\xi_r)$, with $\tau \in [0,1].$ This yields:
\begin{equation}
\label{eq:afterChOfVar}
S_l = \int_{0}^{1} 
u(x,\xi_r + \tau(\xi_{r+1}-\xi_r)) \cdot 
\prod_{i=1}^n (\xi_r + \tau(\xi_{r+1}-\xi_r)-\xi_i) \cdot 
\prod_{l=1}^m \vert \xi_r + \tau (\xi_{r+1}-\xi_r)-\eta_l \vert ^2 \cdot (\xi_{r+1}-\xi_r) \,\dd \tau.
\end{equation}

In order to compute $\val(S_r)$, let us first focus on a factor from \eqref{eq:afterChOfVar} of the form $\xi_r + \tau (\xi_{r+1}-\xi_r)-\xi_i$: 
\begin{itemize}
 \item If $i < r$ or $i > r$, then
 $\xi_r + \tau (\xi_{r+1}-\xi_r)-\xi_i=\xi_r -\xi_i + \tau (\xi_{r+1}-\xi_r)= c_{i,r} x^{e_i} + \hot$
 where $c_{i,r} \neq 0$ and $e_i = \min \{ E(\xi_{r} \wedge \xi_i) , E(\xi_{r} \wedge \xi_{r +1})\}
 = E( \xi_{r} \wedge \xi_{r +1} \wedge \xi_i)$.
 
 \item If $i = r$, we simply have $\xi_r + \tau (\xi_{r+1}-\xi_r)-\xi_i= \tau(\xi_{r+1}-\xi_r)= c_{i,r} x^{e_i} + \hot$ with $e_i = E( \xi_{r} \wedge \xi_{r+1})$. Notice that this is also the valuation of the term $\xi_{r+1}-\xi_r$ coming from the change of variables.
\end{itemize}

Let us focus on a factor of (\ref{eq:afterChOfVar}) of the form $\vert \xi_r + \tau(\xi_{r+1}-\xi_r)-\eta_l\vert ^2$. 
Its valuation, computed in a similar way, is $2 E( \xi_{r} \wedge \xi_{r +1} \wedge \eta_l)$; it is also equal to 
$E(\xi_{r} \wedge \xi_{r +1} \wedge \eta_l) + E(\xi_{r} \wedge \xi_{r +1} \wedge \bar\eta_l)$.

Note that the initial coefficient of each factor of (\ref{eq:afterChOfVar}) is a polynomial function in the variable $\tau$ which is either $> 0$ or $<0$ when $\tau \in (0, 1)$. This shows that the integral on $[0, 1]$ of the initial coefficient of the product is non-zero. Therefore, by an argument similar to that explained in Subsection \ref{ssec:meaninginj}, the valuation of $S_r$ is the sum of the valuations of the factors of \eqref{eq:afterChOfVar} (as $u$ is a unit, its valuation is $0$):
$$\sigma_r 
= \val(S_r)
= \sum_{i=1}^n E(P \wedge \xi_i) \ + \ \sum_{l=1}^m 2E(P \wedge \eta_l) \ + E(P)
= E(P) + \sum_{\gamma \in \mathcal{R}_\Cc(f)} E(P \wedge \gamma),
$$
where $P := \xi_r \wedge \xi_{r+1}$.
\end{proof}

Proposition \ref{prop:sigmai} motivates the following definition, which will play an important role in the statement of Theorem B:
\begin{definition}
 \label{def:intexp}
    The \defi{integrated exponent function} $\sigma: T_{\Rr}(f) \mapsto [0, \infty]$ is defined by:
     $$ \sigma(P) := E(P) + \sum_{\gamma\in \mathcal{R}_{\Cc}(f)} E(P \wedge \gamma)$$
      for every $P \in T_{\Rr}(f)$.
\end{definition}

Note that the integrated exponent function $\sigma$ does not only depend on the pair $(T_{\Rr}(f), E)$, but also on the embedding $T_{\Rr}(f) \subset T_{\Cc}(f)$ and on the multiplicities of the non-real Newton-Puiseux roots of $f$. This is understandable, given that $F_x(y)$ is determined by integration of $f(x,y)$, whose expression (\ref{eq:xiAndEta}) depends both on its real and its non-real Newton-Puiseux roots. 

The following basic property of the integrated exponent function will be an essential ingredient of the proof of Lemma \ref{lem:diffcrit}:
\begin{lemma}
 \label{lem:incareaexp}
      The integrated exponent function is strictly increasing on the poset $(T_{\Rr}(f), \preceq_{T_{\Rr}(f)})$. That is, if $P, Q \in T_\Rr(f)$, then: 
        \[ P \prec_{T_\Rr(f)} Q \implies \sigma(P) < \sigma(Q). \]
\end{lemma}

\begin{proof}
 Assume that $P \prec_{T_\Rr(f)} Q$. This implies that $E(P) < E(Q)$ and that $P \wedge \gamma \preceq_{T_\Rr(f)} Q \wedge \gamma$, therefore $E(P \wedge \gamma) \leq E(Q \wedge \gamma)$ for all $\gamma \in \mathcal{R}_{\Cc}(f)$. By adding all these inequalities, we get:
 $$E(P) + \sum_{\gamma\in \mathcal{R}_{\Cc}(f)} E(P \wedge \gamma) < E(Q) + \sum_{\gamma\in \mathcal{R}_{\Cc}(f)} E(Q \wedge \gamma). $$
 This means exactly that $\sigma(P) < \sigma(Q).$
\end{proof}

%-----------------------------------------------
\subsection{Signs of differences of critical values}
\label{ssec:signsdiff}

Let $P$ be an internal vertex of $T_{\Rr}(f)$ and let $(e_1,e_2,\ldots,e_p)$ be the strictly increasing sequence of outgoing edges of $P$ (see Figure \ref{fig:outGoingEdgesAndAreas}). 

\begin{figure}[H]
\myfigure{0.75}{%
	\begin{tikzpicture}[scale=1]
% define points
\path
  (0,0) coordinate(O)
  ++(2,0) coordinate(A)
  ++(2,0) coordinate(P)
  (P)   
  +(-30:4) coordinate(Q1)
  +(30:4) coordinate(Q2)
  +(-10:4) coordinate(Q3)
   +(10:4) coordinate(Q4)
  +(-50:4) coordinate(Q5)
   +(50:4) coordinate(Q6)
  (Q4) + (2.6,0) coordinate (R1)
  (Q2) + (3,0) coordinate (R2)
;

\fill[red!5] (P) -- +(10:3.5)  arc(10:30:3.5) -- cycle;

% draw lines
\draw[thick, dashed] (O) -- (A);
\draw[thick, dashed] (Q4) -- (R1);
\draw[thick, dashed] (Q2) --  (R2);

\foreach \i in{1,3,5,6}{
  \draw[thick, dashed] (Q\i) -- ++ (2,0);
}
\draw[thick]
  (A) -- (P)
  (P) -- (Q1)
  (P) -- (Q2)
  (P) -- (Q3)
  (P) -- (Q4)
;
\draw[]
  (P) -- (Q5)
  (P) -- (Q6)
;

% draw points
\foreach \v in {P, Q4, Q2}{
  \path  (\v) node[scale=4]{.};
}

% labels
\path (P) node[above left]{$P$};
\path (R1) node[right]{$\xi_{\iota(r)}$};
\path (R2) node[right]{$\xi_{\iota(r)+1}$};
\path (P) ++(-50:2.5) node[above]{$e_1$};
\path (P) ++(-30:2.5) node[above]{$e_2$};
\path (P) ++(-18:2) node[scale=0.9]{$\vdots$};
\path (P)-- ++(52:2.8) node[above]{$e_p$};
\path (P)-- ++(10:2.5) node[above]{$e_{r}$};
\path (P)-- ++(30:2.5) node[above]{$e_{r+1}$};
\path (P) ++(44:2) node[scale=0.9]{$\vdots$};

\path (P)-- ++(20:4) node[]{$S_{\iota(r)}$};

\end{tikzpicture}%
}
\caption{The outgoing edges $e_l$ at the internal vertex $P$ and a basic interval of their totally ordered set.}
\label{fig:outGoingEdgesAndAreas}
\end{figure}
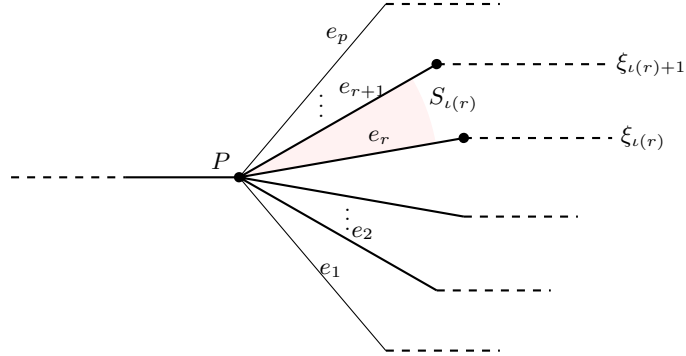

Take $\{e_r,e_{r+1}\} \in BI(\mathcal{E}^+_{T_{\Rr}(f)}(P))$. By Proposition \ref{prop:wedgemap}, there exists a unique basic interval $\{\xi_{\iota(r)},\xi_{\iota(r)+1}\} \in BI(\mathcal{L}(T_{\Rr}(f))) = BI(\mathcal{R}_{\Rr}(f))$, such that $W( \{ \xi_{\iota(r)},\xi_{\iota(r)+1} \}) = \{ e_r,e_{r+1} \}$.
In this way we can associate to the basic interval $ \{ e_r,e_{r+1} \}$ the area function $S_{\iota(r)}$
and also its initial coefficient $s_{\iota(r)} = \lc(S_{\iota(r)})$ (see Equation (\ref{eq:initarea})).

Denote by 
   \[  \sgn : \Rr^* \to \{-1, +1\} \] 
the \defi{sign function}. The following lemma enables to determine the signs of the differences $F_x(\xi_j(x_0)) - F_x(\xi_i(x_0))$ appearing in Formula (\ref{eq:sumSi}) via the knowledge of certain sums of initial coefficients $s_r$ of the area series $S_r$:

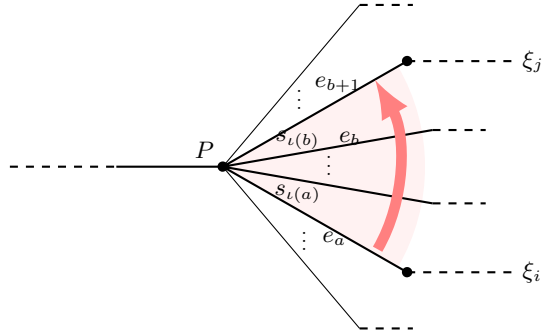
\begin{figure}[H]
\myfigure{0.7}{%
	\begin{tikzpicture}[scale=1]
% define points
\path
  (0,0) coordinate(O)
  ++(2,0) coordinate(A)
  ++(2,0) coordinate(P)
  (P)   
  +(-30:4) coordinate(Q1)
  +(30:4) coordinate(Q2)
  +(-10:4) coordinate(Q3)
   +(10:4) coordinate(Q4)
  +(-50:4) coordinate(Q5)
   +(50:4) coordinate(Q6)
  (Q1) + (2,0) coordinate (R1)
  (Q2) + (2,0) coordinate (R2)
;

\fill[red!5] (P) -- +(-30:3.8)  arc(-30:30:3.8) -- cycle;

% draw lines
\draw[thick, dashed]
  (O) -- (A)
;

\foreach \i in{1,2}{
  \draw[thick, dashed] (Q\i) -- ++ (2,0);
}
\foreach \i in{3,...,6}{
  \draw[thick, dashed] (Q\i) -- ++ (1,0);
}

\draw[thick]
  (A) -- (P)
  (P) -- (Q1)
  (P) -- (Q2)
  (P) -- (Q3)
  (P) -- (Q4)
;
\draw[]
  (P) -- (Q5)
  (P) -- (Q6)
;

% draw points
\foreach \v in {P, Q1, Q2}{
  \path  (\v) node[scale=4]{.};
}

% labels
\path (P) node[above left]{$P$};
\path (R1) node[right]{$\xi_i$};
\path (R2) node[right]{$\xi_j$};
\path (P)-- ++(-20:1.5) node[]{$s_{\iota(a)}$};
\path (P)-- ++(-33:2.5) node[]{$e_{a}$};
\path (P)-- ++(20:1.5) node[]{$s_{\iota(b)}$};
\path (P)-- ++(30:2.5) node[above]{$e_{b+1}$};
\path (P)-- ++(6:2.4) node[above]{$e_{b}$};
\path (P)-- ++(4:2) node[scale=0.8]{$\vdots$};
\path (P)-- ++(-40:2) node[scale=0.8]{$\vdots$};
\path (P)-- ++(44:2) node[scale=0.8]{$\vdots$};

\draw[->,>=latex,red!50, line width=4pt, shorten <=2pt] ($(P) + (-30:3.3)$) arc (-30:30:3.3);

\end{tikzpicture}%
}
\caption{Outgoing edges between two leaves.}
\label{fig:ea}
\end{figure}

\begin{lemma}
\label{lem:diffcrit} 
 Assume that $f$ is right-reduced. Let $\xi_i \order_{\Rr} \xi_j$ be two real roots of $f$. Consider the vertex $P := \xi_i \wedge \xi_j$ of the real contact tree $T_{\Rr}(f)$. Let $s_{\iota(a)}$, $s_{\iota(a+1)}$, \ldots, $s_{\iota(b)}$ be the initial coefficients associated to the outgoing edges from $P$, in between the edges going to the leaves $\xi_i$ and $\xi_j$ (see Figure \ref{fig:ea}). If $s_{\iota(a)} + s_{\iota(a+1)} + \ \cdots \ + s_{\iota(b)} \neq 0$, then there exists $\epsilon > 0$ such that:
$$\sgn\big( F_{x_0}(\xi_j(x_0)) - F_{x_0}(\xi_i(x_0)) \big)
= \sgn\big( s_{\iota(a)} + s_{\iota(a+1)} + \ \cdots \ + s_{\iota(b)} \big)$$
for every $x_0 \in (0, \epsilon]$.
\end{lemma}

\begin{proof}
The statement is equivalent to the fact that the sign of the initial coefficient of $F_x(\xi_j) - F_x(\xi_i)$ is equal to $\sgn\big( s_{\iota(a)} + s_{\iota(a+1)} + \ \cdots \ + s_{\iota(b)} \big)$.
In order to prove this property, we use the fact that for every $r \in \{1, \dots, n-1\}$, the valuation $\sigma_r$ of $S_r$ only depends on the vertex $P := \xi_r \wedge \xi_{r+1}$ (see Proposition \ref{prop:sigmai}) and that it is a strictly increasing function on the real contact tree (see Lemma \ref{lem:incareaexp}). Therefore:
\begin{align*}
 F_x(\xi_j) - F_x(\xi_i)
 & = S_i + S_{i+1} + \cdots + S_{j-1} \\
 &= \sum_{\substack{i \leq r < j \\ 
 P = \xi_r \wedge \xi_{r+1}}} S_{r} + 
 \sum_{\substack{i \leq r < j \\ 
 P \prec_{T_{\Rr}(f)} \xi_r \wedge \xi_{r+1}}} S_{r} \\
 &= \big( s_{\iota(a)} + s_{\iota(a+1)} + \cdots + s_{\iota(b)}\big) x^{\sigma(P)} \ + \hot .
\end{align*}
One concludes using the non-vanishing hypothesis $s_{\iota(a)} + s_{\iota(a+1)} + \ \cdots \ + s_{\iota(b)} \neq 0$.
\end{proof}

%%%%%%%%%%%%%%%%%%%%%%%%%%%%%%%%%%%%%%%%%%%%%%%%
\section{The injectivity condition and the combinatorial types of morsifications}
\label{sec:theorem}

%-----------------------------------------------
We start this section by defining the \emph{injectivity condition} on real contact trees of right-reduced series, which is a crucial hypothesis for our main Theorems A and B. Then we define a second planar structure on those contact trees under the hypothesis that the injectivity condition is satisfied: the \emph{integrated planar structure}. Finally, we state and prove Theorem A, which describes the combinatorial types of morsifications whenever the injectivity condition is satisfied.

%-----------------------------------------------
\subsection{The injectivity condition}
\label{ssec:injCond}

We explained in Subsection \ref{ssec:meaninginj} how our valuation-theoretical approach leads naturally to the \emph{injectivity condition}. In the present subsection we formulate it in a way which explains its name. The equivalence of this formulation and that of Subsection \ref{ssec:meaninginj} results from Lemma \ref{lem:diffcrit}.

We keep the notations $\iota(k)$ introduced in Subsection \ref{ssec:signsdiff}.
The \emph{injectivity condition} on $f$ will involve all the sums of initial coefficients $s_{\iota(k)}$ taken on consecutive basic intervals of $\mathcal{E}^+_{T_{\Rr}(f)}(P)$: we impose that all these sums are non-zero. 
This may be also expressed as the condition that the following \defi{discrete integration map} at $P$ is injective:

\begin{equation}
\label{eq:dim}
\begin{array}{clcl}
{\smallint}_P : & \{0,1,\ldots,p-1\}& \longrightarrow & \Rr \\
& 0 & \longmapsto & 0 \\
& \cdots & & \\
& i & \longmapsto & s_{\iota(1)} + s_{\iota(2)} + \cdots + s_{\iota(i)} \\
& \cdots & & \\
& p-1 & \longmapsto & s_{\iota(1)} + s_{\iota(2)} + \ \cdots \ + s_{\iota(p-1)} \\
\end{array}
\end{equation}

\begin{definition} 
 \label{def:injcond}
 The \defi{injectivity condition} on the right-reduced series $f \in \Rr\{x,y\}$ is:
\begin{equation}
\emph{\text{For each internal vertex $P$, the discrete integration map ${\smallint}_P$ is injective.}}
\label{Inj}
\tag{$\mathcal{I}nj$}
\end{equation}
\end{definition}

The injectivity condition is equivalent to:
$$ \mbox{\em For each internal vertex } P, \mbox{\em each partial sum } s_{\iota(a)} + s_{\iota(a+1)} + \ \cdots \ + s_{\iota(b)} 
 \mbox{ \em of consecutive terms is non-zero.} $$

\begin{example}
\label{ex:bintree}
The injectivity condition is automatically satisfied when the rooted tree $T_{\Rr}(f)$ is {\em binary}, that is, when all its internal vertices have valency $3$ (as in \cite{sorea_portugaliae}). 
\end{example}

\begin{example}
\label{ex:symtree}
The injectivity condition is never satisfied for series $f$ which are odd in the variable $y$ (that is, such that $f(x, -y) = - f(x,y)$) and verify $\mathrm{ord}_y(f(0, y)) \geq 3$. Indeed, in this case $F(x, -y) = F(x,y)$, which implies that $F_{x_0}: [-h, h] \to \Rr$ is even and has at least three critical points for every Morse rectangle $[0, \epsilon] \times [-h, h]$ of $F_x(y)$ and every $x_0 \in (0, \epsilon]$. Therefore, $F_{x_0}$ is not Morse.  For instance, the injectivity condition is not satisfied if $f(x,y) = y(y^2 - x^3)$.
\end{example}

\begin{example} 
\label{ex:notnec}
The injectivity condition is not necessary for $F_x(y)$ to be a morsification. Consider for instance $f(x,y) := (y+x)(y+x^2)y(y-x^2)(y-x-cx^2)$. The corresponding five Newton-Puiseux roots $\xi_i$ are all real, with $(\xi_1, \ldots, \xi_5) = (- x, - x^2, 0, x^2, x + c x^2)$.
Elementary computations prove that for all sufficiently small $c>0$, and for sufficiently small $x_0>0$, $y \mapsto F_{x_0}(y)$ is a Morse function. However $f$ \emph{does not satisfy} the injectivity condition since:
$$
S_1 = \frac{1}{12}x^6 + \hot, \qquad
S_2 = -\frac{1}{4}x^{10} + \hot, \qquad
S_3 = \frac{1}{4}x^{10} + \hot, \qquad
S_4 = -\frac{1}{12}x^6 + \hot,$$
the areas $S_1$ and $S_4$ satisfy $\sigma_1=\sigma_4 = 6$, but the sum $S_1+S_2+S_3+S_4$ has valuation greater than $6$.
\end{example}

%-----------------------------------------------
\subsection{Integrated contact trees}

Recall from Proposition \ref{prop:realtotplan} that the real total order $\order_{\Rr}$ on the set $\mathcal{R}_{\Rr}(f)$ of leaves of $T_{\Rr}(f)$ is planar relative to the abstract tree $T_{\Rr}(f)$. We will define now a second total order $\order_{\Int}$ on the set $\mathcal{R}_{\Rr}(f)$, whenever $f$ satisfies the injectivity condition.
By contrast with the real total order, the total order $\order_{\Int}$ will not be defined directly on the set of leaves, but it will be associated to a planar structure in the sense of Definition \ref{def:plastr}.

Let $\xi_i$, $\xi_j$ (with $i<j$) be two real roots of $f$. Let $P := \xi_i \wedge \xi_j$ and $s_{\iota(a)}$, $s_{\iota(a+1)}$, \ldots, $s_{\iota(b)}$ denote the initial coefficients %$s_k$ 
associated with the outgoing edges at $P$, in between the edges going to the leaves $\xi_i$ and $\xi_j$.
Then one may define a binary relation $\order_{\Int, P}$ on the set of outgoing edges at $P$ by:
$$\xi_i \order_{\Int, P} \xi_j \iff s_{\iota(a)} + \cdots + s_{\iota(b)} > 0.$$
In terms of the discrete integration map of formula (\ref{eq:dim}), this equivalence may be reformulated as follows:
$$\xi_i \order_{\Int, P} \xi_j \iff {\smallint}_P(e_a) < {\smallint}_P(e_{b+1})$$
if $e_{a}$ is the outgoing edge going from $P$ to $\xi_i$ and $e_{b+1}$ is the outgoing edge going from $P$ to $\xi_j$ (see Figure \ref{fig:ea}). 

The fact that this binary relation is a strict total order results from the injectivity condition (\ref{Inj}). The set of these total orders, when $P$ varies among the internal vertices of $T_{\Rr}(f)$, defines a planar structure. Therefore, there is an induced total order $\order_{\Int}$ on the set $\mathcal{R}_{\Rr}(f)$ of leaves of $T_{\Rr}(f)$.

\begin{definition} 
\label{def:integratedtree}
Assume that $f$ satisfies the injectivity condition (\ref{Inj}). The collection of all total orders $\order_{\Int, P}$, when $P$ varies among the internal vertices of $T_{\Rr}(f)$, is the \defi{integrated planar structure on $T_{\Rr}(f)$}. We say that the abstract rooted tree $T_{\Rr}(f)$ endowed with this planar structure is the \defi{integrated contact tree} $T_{\Int}(f)$ of $f$. The associated total order $\order_{\Int}$ on the set $\mathcal{R}_{\Rr}(f)$ of leaves of $T_{\Rr}(f)$ is the \defi{integrated order}.
\end{definition}

The attribute ``integrated'' in the previous definition is motivated by the fact that the integrated orders are defined using the \emph{discrete integration maps} ${\smallint}_P$ of formula (\ref{eq:dim}).

%-----------------------------------------------
\subsection{The first main theorem}

We get from Definition \ref{def:integratedtree} a pair $(\order_{\Rr},\order_{\Int})$
of total orders on the set $\mathcal{R}_{\Rr}(f)$ of real roots of $f$, obtained by identifying it with the set of leaves of $T_{\Rr}(f)$, seen as a planar tree in two ways. This allows us to formulate our first main theorem, describing the combinatorial types (in the sense of Definition \ref{def:combinmorsif}) of the primitives of right-reduced series which satisfy the injectivity condition:

\begin{theoremA}
\label{th:AA}
    Assume that the series $f\in\mathbb{R}\{x,y\}$ is right-reduced and satisfies the injectivity condition \eqref{Inj}. Consider a primitive $F_x(y) \in \Rr \{x, y\}$ of $f$. Then $F_x(y)$ is a morsification and its combinatorial type is represented by the bi-ordered set $(\mathcal{R}_{\Rr}(f),\order_{\Rr},\order_{\Int})$.
\end{theoremA}

\begin{proof}
By the injectivity condition, the series $F_x(\xi_i)$ are pairwise distinct when $\xi_i$ varies among the real Newton-Puiseux roots of $f$. Hence, by Proposition \ref{prop:morsif}, $F_x(y)$ is a morsification. 

Let $[0, \epsilon] \times I$ be a Morse rectangle for $F_x(y)$ (see Definition \ref{def:morsif}). 
We will prove that $(\mathcal{R}_{\Rr}(f),\order_{T},\order_{\Int})$ and $(\Crit(F_{x_0}),\order_s,\order_t)$ are isomorphic as bi-ordered sets, for every $x_0 \in (0, \epsilon]$.
Consider the bijection from $\mathcal{R}_{\Rr}(f)$ to $\Crit(F_{x_0})$,
sending a real Newton-Puiseux root $\xi_i$ of $f$ to the element $(\xi_i(x_0), F_{x_0}(\xi_i(x_0)))$ of the critical graph of $F_{x_0}$.
The orders $\order_{\Rr}$ and $\order_s$ correspond by the previous bijection:
recall that $(\xi_i(x_0), F_{x_0}(\xi_i(x_0))) \order_s (\xi_j(x_0), F_{x_0}(\xi_j(x_0)))$ iff 
$\xi_i(x_0) \order \xi_j(x_0)$, which is equivalent to $\xi_i \order_{\Rr} \xi_j$.

It remains to prove
that the orders $\order_t$ and $\order_{\Int}$ also correspond by the bijection. 
Let $i<j$ such that 
$(\xi_i(x_0), F_{x_0}(\xi_i(x_0))) \order_t (\xi_j(x_0), F_{x_0}(\xi_j(x_0)))$. 
This means that 
$F_{x_0}(\xi_i(x_0)) < F_{x_0}(\xi_j(x_0))$ whenever $x_0 \in (0, \epsilon]$. 
By Lemma \ref{lem:diffcrit}, this is equivalent to the inequality
$s_{i_a} + s_{i_{a+1}} + \ \cdots \ + s_{i_{b}} > 0$. 
By the definition of the integrated total order $\order_{\Int}$, we get indeed the inequality $\xi_i \order_{\Int} \xi_j$.
\end{proof}

\begin{example}
\label{ex:parabola1}
We consider again the example $f(x,y) = 3(y^2-x)$ of Subsection \ref{ssec:fundpict}. The Newton-Puiseux roots of $f$ are $\xi_1 = -x^{1/2}$ and $\xi_2 = x^{1/2}$. The right semi-branches corresponding to $\xi_1$ and $\xi_2$, one below, one above the $x$-axis, are depicted on the left of Figure \ref{fig:parabola} (see also Figure \ref{fig:fundamental}). The contact tree $T_{\Rr}(f) = T_{\Rr}(\xi_1,\xi_2)$ is depicted in the center. 

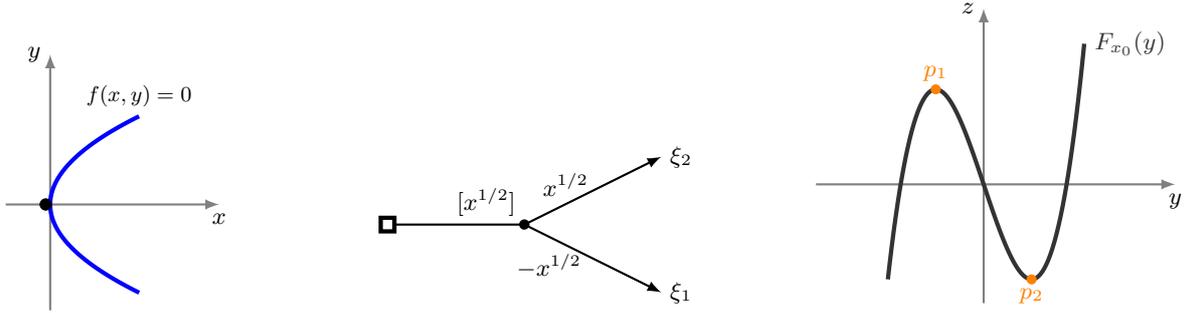
\begin{figure}[H]
 \myfigure{0.9}{
	\begin{tikzpicture}[scale=1.3]
\draw [->, >=latex, gray, thick](-0.5,0) -- (1.9,0) node[below,black] {$x$};;
\draw [->, >=latex, gray, thick] (0,-1.2)--(0,1.7) node[left,black] {$y$};

\draw[ultra thick, color=blue,domain=-1:1,samples=50,smooth] plot ({(\x)^2},{\x}) node[above,scale=0.9,black] {$f(x,y)=0$};

%\node [below, color=black] at (-0.4,0) {$O$};

\node[draw,circle, inner sep=1.5pt,color=black, fill=black] at (-0.05,0){};
\end{tikzpicture}%

	\begin{tikzpicture}[scale=1]
% define points
\path
  (0,0) coordinate(O)
  ++(2,0) coordinate(A)
  ++(2,0) coordinate(P)
  (P)   
  +(2,-1) coordinate(Q1)
  +(2,1) coordinate(Q2)
  (Q1) + (2,0) coordinate (R1)
  (Q2) + (2,0) coordinate (R2)
;

% draw lines
%\draw[thick, dashed]
  %(O) -- (A)
  %(Q1) -- (R1)
  %(Q2) -- (R2)
%;
\draw[thick]
  (A) -- (P);
  \draw[->,>=latex,thick]
  (P) -- (Q1);
  \draw[->,>=latex,thick]
  (P) -- (Q2)
;
% draw points
\foreach \v in { A, P}{
  \path  (\v) node[scale=4]{.};
}
\filldraw[ultra thick,fill=white] (A) ++ (-0.1,-0.1) rectangle ++(0.2,0.2);
% labels
\path (P) node[above left]{$[x^{1/2}]$};
\path (Q1) node[right]{$\xi_1$};
\path (Q2) node[right]{$\xi_2$};
\path (P)--(Q1) node[below,pos=0.3]{$ -x^{1/2}\ \ \ \ $};
\path (P)--(Q2) node[above,pos=0.3]{$ x^{1/2}$};
\end{tikzpicture}%

	\begin{tikzpicture}[scale=0.7]
% Axes
\draw[->,>=latex,thick, gray] (-3.5,0)--(4,0) node[below,black] {$y$};
\draw[->,>=latex,thick, gray] (0,-2.5)--(0,3.7) node[left,black] {$z$};

\def\c{1.0}
\draw[ultra thick, color=black!80,domain=-2:2.1,samples=100] plot (\x,{\x^3-3*\c*\x)}) node[right] {$F_{x_0}(y)$};

\def\racc{sqrt{\c}}
\coordinate (P1) at (-\racc, 2*\c*\racc );	
\coordinate (P2) at (\racc,-2*\c*\racc);	

\fill[orange] (P1) circle (3pt) node[above]{$p_1$};
\fill[orange] (P2) circle (3pt) node[below]{$p_2$};

\end{tikzpicture}%

 }
\caption{The Newton-Puiseux roots $\xi_1$, $\xi_2$ from Example \ref{ex:parabola1} (left), the contact tree $T_{\Rr}(f)$ (center) and the graph of the primitive $F_{x_0}(y)=y^3-3x_{0}y$, with $x_0 > 0$ (right).}
\label{fig:parabola}
\end{figure}
 
The chosen primitive is $F_x(y)=y^3-3xy$. Thus we obtain $S_1 = F_x(\xi_2) - F_x(\xi_1) =-4 x^{3/2}$, hence $s_1=-4$ and $\sigma(\xi_1\wedge \xi_2)=3/2$. 

Fix $x_0>0$ and denote by $p_1 = (\xi_1(x_0), F_{x_0}(\xi_1(x_0)))$ and $p_2 = (\xi_2(x_0), F_{x_0}(\xi_2(x_0)))$ the elements of the critical graph of $F_{x_0}$ (see the right of Figure \ref{fig:parabola}). We have $p_1 \order_s p_2$. As $S_1(x_0) <0$, then $F_{x_0}(\xi_2(x_0)) < F_{x_0}(\xi_1(x_0))$, therefore $p_2 \order_t p_1$.
On the other hand $\xi_1 \order_{\Rr} \xi_2$ and since $s_1<0$, we have $\xi_2 \order_{\Int} \xi_1$.
Conclusion: the bi-orders on the critical graph and on the set of real roots of $f$ are isomorphic.
\end{example}

\begin{remark}  
\label{rem:strongconstr} 
    As a consequence of Theorem A, the combinatorial type of the primitives $F_x(y)$ of a right-reduced series $f(x,y)$ which satisfies the injectivity condition is constrained by the structure of the real contact tree $T_{\Rr}(f)$ of $f$. For instance, if $T_{\Rr}(f)$ is isomorphic to the planar tree with three leaves $\ell_1, \ell_2, \ell_3$ from Example \ref{ex:planarRelativeToT}, then the combinatorial type of $F_x(y)$ cannot be 
    $(\ell_1 <_1 \ell_2 <_1 \ell_3, \ \ell_2 <_2 \ell_1 <_2 \ell_3)$. 
\end{remark}

\begin{remark}
\label{rem:notinj}
    Even if the right-reduced series does not satisfy the injectivity condition, Lemma \ref{lem:diffcrit} allows to get constraints on the combinatorial types of its primitives $F_x(y)$, when these primitives are morsifications. More precisely, for each pair of real roots $\xi_i \order_{\Rr} \xi_j$, the lemma gives the order relation of the critical values $F_{x_0}(\xi_j(x_0))$ and $F_{x_0}(\xi_i(x_0))$ for $x_0$ small enough, whenever the sum $s_{\iota(a)} + s_{\iota(a+1)} + \ \cdots \ + s_{\iota(b)}$ is non-zero. For instance, this sum is non-zero if in the planar tree $T_{\Rr}(f)$ there is no other outgoing edge at $\xi_i \wedge \xi_j$ in between the edges going to the leaves $\xi_i$ and $\xi_j$. Indeed, then the sum above contains only one term, which is by definition non-zero. 
\end{remark}

%%%%%%%%%%%%%%%%%%%%%%%%%%%%%%%%%%%%%%%%%%%%%%%%
\section{The contact tree of the apparent contour in the target}
\label{sec:contour}

Assume again that the series $f\in\mathbb{R}\{x,y\}$ is right-reduced, satisfies the injectivity condition \eqref{Inj}, and that $F_x(y) \in \Rr\{x,y\}$ denotes a primitive of $f$.
In this section, we identify the real contact tree of the apparent contour in the target of the morphism $(x,y) \to (x, F_x(y))$ with the integrated contact tree of $f$ from Definition \ref{def:integratedtree} (see Theorem B).

%-----------------------------------------------
\subsection{Real polar and discriminant curves}
 
By Theorem A, $F_x(y)$ is a morsification. Let $[0, \epsilon] \times I$ be a Morse rectangle for it. 
Denote as before by $\xi_1, \dots, \xi_n$ the real Newton-Puiseux roots of $f$. We will consider in full generality three geometric objects which appeared already in Figure \ref{fig:fundamental} of the introduction:
\begin{itemize}
 \item The graph of the function $(x,y) \mapsto F_x(y)$, that is, the surface: 
 $$\mathcal{G} := \big\{ (x,y, F(x,y)) \mid (x,y) \in [0, \epsilon] \times I \big\} \subset \Rr^3.$$
 
 \item The projection $\pi : \mathcal{G} \to \Rr^2_{x,z}$ given by $\pi(x,y,z) = (x,z)$.
 The critical image $\Delta \subset \Rr^2_{x,z}$ of $\pi$ is called the \defi{discriminant curve} or the \defi{apparent contour in the target} of $\pi$.
 The real Newton-Puiseux roots of $\Delta$ in the coordinate system $(x,z)$ are denoted by $\delta_1, \ldots, \delta_n \in \Rr\{x^{\frac{1}{\Nn}}\}$.

 \item The \defi{polar curve} $\Gamma \subset{\Rr^2_{x,y}}$ of $\pi$ is the projection of the critical locus of $\pi$ to the horizontal real plane $\Rr^2_{x,y}$. Then $\Gamma$ is defined by $f(x,y) =0$ and its real Newton-Puiseux roots are exactly $\xi_1,\ldots,\xi_n$.
\end{itemize}
By construction we have, for all $i=1,\ldots,n$:
\begin{equation*}
\label{eq:xidelta}
\delta_i = F_x\big(\xi_i \big).
\end{equation*}

\begin{remark}
\label{rem:absisom}
Note that both $\Gamma$ and $\Delta$ are semi-analytic germs. As a consequence of  Theorem B below, their real contact trees are isomorphic as abstract rooted trees whenever the injectivity condition is satisfied.
\end{remark}

In the real plane $\Rr^2_{x,y}$, the right semi-branches $\Gamma_{\xi_1},\ldots, \Gamma_{\xi_n}$ are ordered by the real total order $\order_{\Rr}$ of Definition \ref{def:realtot}.
Our goal is to determine the total order of their projections $\Gamma_{\delta_1},\ldots, \Gamma_{\delta_n}$ in the plane $\Rr^2_{x,z}$. This order is encoded in the contact tree $T_\Rr(\delta_1,\ldots,\delta_n)$ of the real Newton-Puiseux roots of the apparent contour in the target $\Delta$. Theorem B below describes the isomorphism type of this planar tree.

%-----------------------------------------------
\subsection{The second main theorem}

Our second main theorem shows that the planar tree $T_{\Int}(f)$ from Definition \ref{def:integratedtree} is isomorphic to a real contact tree:
 
\begin{theoremB}
\label{th:BB}
Let $f \in \Rr\{x, y\}$ be a right-reduced series satisfying the injectivity condition \eqref{Inj}. The integrated contact tree $T_{\Int}(f)$ is isomorphic to the real contact tree $T_\Rr(\delta_1,\ldots,\delta_n)$ of the real roots $\delta_i = F_x(\xi_i)$ of the apparent contour in the target of the projection $\pi$.
\end{theoremB}
 
\begin{proof}
Since $f$ is right-reduced, the Newton-Puiseux series $\xi_i$ are pairwise distinct.
The injectivity condition implies that the Newton-Puiseux series $\delta_i$ are also pairwise distinct. 

We will prove that there exists a unique homeomorphism from $T_{\Int}(f)$ to $T_\Rr(\delta_1,\ldots,\delta_n)$ which respects the labels in $\{1, \dots, n\}$ of their leaves and sends the integrated exponent function of $T_{\Int}(f)$ to the exponent function of $T_\Rr(\delta_1,\ldots,\delta_n)$. Its \emph{uniqueness} comes from the fact that the constraint of respecting the labels obliges to identify the segments $[O, \xi_i]$ and $[O, \delta_i]$ for every $i \in \{1, \dots, n\}$, and that there is only one such identification which transforms the integrated exponent function on $[O, \xi_i]$ into the exponent function on $[O, \delta_i]$. It is therefore enough to prove the \emph{existence} of such a homeomorphism.

As an abstract rooted tree, $T_{\Int}(f)$ coincides by construction with $T_{\Rr}(f) = T_\Rr(\xi_1,\ldots,\xi_n)$. We first prove that $T_\Rr(\xi_1,\ldots,\xi_n)$ and $T_\Rr(\delta_1,\ldots,\delta_n)$ are homeomorphic as abstract rooted trees, by a homeomorphism which respects the constraints above.
\begin{itemize}
 \item For every $i \in \{1, \dots, n\}$, there exists a unique homeomorphism from $[O, \xi_i] \hookrightarrow T_\Rr(\xi_1,\ldots,\xi_n)$ to $[O, \delta_i] \hookrightarrow T_\Rr(\delta_1,\ldots,\delta_n)$, which identifies the restrictions to those segments of the integrated  exponent function $\sigma$ on $T_\Rr(\xi_1,\ldots,\xi_n)$ and of the exponent function $E$ on $T_\Rr(\delta_1,\ldots,\delta_n)$. This is a consequence of Lemma \ref{lem:incareaexp} and of the fact that, by Definition \ref{def:intexp}, the restrictions $[O, \xi_i] \to [0, \infty]$ of the function $\sigma$ are increasing homeomorphisms.
 
 \item Let us consider distinct series $\xi_i$ and $\xi_j$, which are represented by two leaves of the tree $T_\Rr(\xi_1,\ldots,\xi_n)$.  
 Denote $P:= \xi_i \wedge \xi_j$. We have $E(P) = \val(\xi_j - \xi_i)$.
 The series $\delta_i := F_x(\xi_i)$ and $\delta_j := F_x(\xi_j)$ are two leaves of the tree 
 $T_\Rr(\delta_1,\ldots,\delta_n)$. Let $P' := \delta_i \wedge \delta_j$. In the tree $T_\Rr(\delta_1,\ldots,\delta_n)$, we have $E(P') = \val(\delta_j - \delta_i) = \sigma(P)$ (see Proposition \ref{prop:sigmai}). Therefore, the previous homeomorphisms glue into a homeomorphism from $T_\Rr(\xi_1,\ldots,\xi_n)$ to $T_\Rr(\delta_1,\ldots,\delta_n)$.
 
\end{itemize}

To prove that this homeomorphism identifies the planar structures of $T_{\Int}(f)$ and $T_\Rr(\delta_1,\ldots,\delta_n)$, we need to prove that the leaves are ordered in the same way:
$$\delta_i \order_{T_\Rr(\delta_1,\ldots,\delta_n)} \delta_j \iff \xi_i \order_{\Int} \xi_j.$$
Recall that, to decide if $\xi_i \order_{\Int} \xi_j$, we look at the ``local order'' at $P = \xi_i \wedge \xi_j$. Denote by $e_a$ the outgoing edge from $P$ to $\xi_i$ and by $e_{b+1}$ the outgoing edge from $P$ to $\xi_j$. Then $\xi_i \order_{\Int} \xi_j \iff e_a \order_{\Int} e_{b+1}$.
Now:
\begin{align*}
\delta_i \order_{T_\Rr(\delta_1,\ldots,\delta_n)} \delta_j 
 &\iff F_x(\xi_i ) \order_{\Rr} F_x(\xi_j ) \qquad \text{(ordered by the real order as series in $x$)} \\
 &\iff s_{i_a}+s_{i_{a+1}} + \cdots + s_{i_b} > 0 \qquad \text{(by Lemma \ref{lem:diffcrit})} \\
 &\iff e_a \order_{\Int} e_{b+1} \qquad \text{(by Definition \ref{def:integratedtree} of the integral order)} \\
 &\iff \xi_i \order_{\Int} \xi_j.
\end{align*}

\end{proof}

\begin{figure}[H]
\myfigure{0.9}{
	\begin{tikzpicture}[scale=1.2]
\draw [->, >=latex, gray, thick](-1,0) -- (2,0) node[below,black] {$x$};
\draw [->, >=latex, gray, thick] (0,-1.2)--(0,1.4) node[left,black] {$z$};

\draw[ultra thick, color=red!80,domain=0:1,samples=50,smooth] plot ({\x^2},{\x^3}) node[right] {$\Gamma_{\delta_1}$};
\draw[ultra thick, color=red!80,domain=0:1,samples=50,smooth] plot ({\x^2},{-\x^3}) node[right] {$\Gamma_{\delta_2}$};

%\node [below, color=black] at (-0.4,0) {$O$};

\node[draw,circle, inner sep=1.5pt,color=black, fill=black] at (-0.05,0){};
\end{tikzpicture}%

	\begin{tikzpicture}[scale=0.8]
% define points
\path
  (0,0) coordinate(O)
  ++(2,0) coordinate(A)
  ++(2,0) coordinate(P)
  (P)   
  +(2,-1) coordinate(Q1)
  +(2,1) coordinate(Q2)
  (Q1) + (2,0) coordinate (R1)
  (Q2) + (2,0) coordinate (R2)
;

% draw lines
%\draw[thick, dashed]
  %(O) -- (A)
  %(Q1) -- (R1)
  %(Q2) -- (R2)
%;
\draw[thick]
  (A) -- (P);
  \draw[->,>=latex,thick]
  (P) -- (Q1);
  \draw[->,>=latex,thick]
  (P) -- (Q2)
;
% draw points
\foreach \v in { A, P}{
  \path  (\v) node[scale=4]{.};
}
\filldraw[ultra thick,fill=white] (A) ++ (-0.1,-0.1) rectangle ++(0.2,0.2);
% labels
\path (P) node[above left]{$[x^{3/2}]$};
\path (Q1) node[right]{$\delta_2$};
\path (Q2) node[right]{$\delta_1$};
\path (P)--(Q1) node[right, pos=0.25]{$\  -2x^{3/2}$};
\path (P)--(Q2) node[above,pos=0.3]{$\ 2x^{3/2}$};

\node at (4,-2) {$T_\Rr(\delta_1,\delta_2)$};
%\draw[->,>=latex,thick] (P) -- (Q1);

\end{tikzpicture}%

	\begin{tikzpicture}[scale=0.8]
% define points
\path
  (0,0) coordinate(O)
  ++(2,0) coordinate(A)
  ++(2,0) coordinate(P)
  (P)   
  +(2,-1) coordinate(Q1)
  +(2,1) coordinate(Q2)
  (Q1) + (2,0) coordinate (R1)
  (Q2) + (2,0) coordinate (R2)
;

% draw lines
%\draw[thick, dashed]
  %(O) -- (A)
  %(Q1) -- (R1)
  %(Q2) -- (R2)
%;
\draw[thick]
  (A) -- (P);
  \draw[->,>=latex,thick]
  (P) -- (Q1);
  \draw[->,>=latex,thick]
  (P) -- (Q2)
;
% draw points
\foreach \v in { A, P}{
  \path  (\v) node[scale=4]{.};
}
\filldraw[ultra thick,fill=white] (A) ++ (-0.1,-0.1) rectangle ++(0.2,0.2);
% labels
\path (P) node[above left]{$[x^{1/2}]$};
\path (Q1) node[right]{$\xi_2$};
\path (Q2) node[right]{$\xi_1$};
%\path (P)--(Q1) node[below,pos=0.3]{$ x^{\frac{1}{2}}$};
%\path (P)--(Q2) node[above,pos=0.3]{$- x^\frac{1}{2}$};

\node at (4,-2) {$T_{\Int}(f)$};
\end{tikzpicture}%

}
\caption{The right semibranches $\Gamma_{\delta_1}$ and $\Gamma_{\delta_2}$ from Example \ref{ex:parabola2} (left); the isomorphic planar trees $T_\Rr(\delta_1,\delta_2)$ (center) and $T_{\Int}(\xi_1,\xi_2)$ (right).}
\label{fig:OneCusp}
\end{figure}
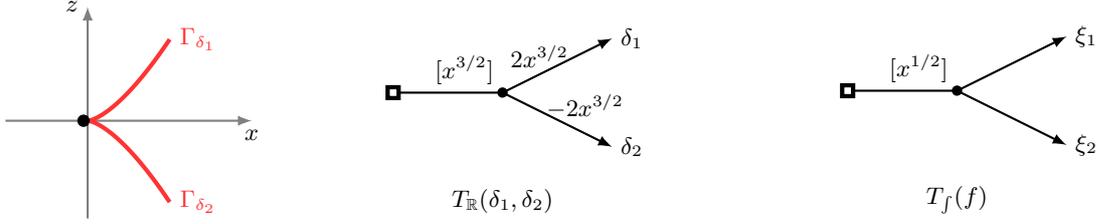

\begin{example}
\label{ex:parabola2}
Revisiting Example \ref{ex:parabola1}, we have $\delta_1 =F_x(\xi_1)=F_x(-x^{1/2})=2x^{3/2}.$ Similarly, $\delta_2 =-2x^{3/2}$. The real contact tree $T_{\Rr}(\delta_1,\delta_2)$ is depicted in Figure \ref{fig:OneCusp}. We have $\delta_2<_{T_\Rr(\delta_1,\delta_2)} \delta_1$. Recall from Example \ref{ex:parabola1} that $\xi_2 \order_{\Int} \xi_1$. Hence, the planar trees $T_\Rr(\delta_1,\delta_2)$ and $T_{\Int}(\xi_1,\xi_2)$ are isomorphic (we do not take into account the labels of the internal vertices).
\end{example}

%%%%%%%%%%%%%%%%%%%%%%%%%%%%%%%%%%%%%%%%%%%%%%%%
\section{An example with three cusps}
\label{sec:example}

Let us consider the following $f\in\Rr\{x,y\}$: 
\begin{equation*}
% \label{eq:threecusps}
 f(x,y):=(y^2-x^3)(y^2-c^2 x^3)(y^3-x^2),
\end{equation*}
where $c>1$ is a parameter. Its Newton-Puiseux roots are (here $\rho = e^{2\ii\pi/3}$): 
$$
\renewcommand{\arraystretch}{1.3}
\begin{array}{l}
\xi_1 =-c x^{3/2} \\
\xi_2=- x^{3/2} \\ 
\xi_3=x^{3/2} \\
\xi_4=c x^{3/2} \\
\xi_5=x^{2/3} \\
\end{array}
\qquad\qquad\qquad
\begin{array}{l}
\eta=\rho \ x^{2/3} \\
\bar{\eta}=\rho^2 \ x^{2/3} \\
\end{array}
$$
There are therefore $5$ real right semi-branches $\Gamma_{\xi_i}$ and $2$ non-real ones (which we may define similarly to the real ones, as the germs $\Gamma_{\eta}$ and $\Gamma_{\bar\eta}$ of the graphs of $x \mapsto \eta$ and $x \mapsto \bar{\eta}$, where $x \geq 0$), as represented in Figure \ref{fig:ThreeCusps}.

\begin{figure}[H]
\myfigure{0.8}{
	\begin{tikzpicture}[scale=6]
\draw [->, >=latex, gray, thick](-0.3,0) -- (0.55,0) node[below,black] {$x$};
\draw [->, >=latex, gray, thick] (0,-0.4)--(0,0.5) node[left,black] {$y$};

\def\k{1.5}

\draw[ultra thick, color=blue!50,domain=0:0.6,samples=50,smooth] plot ({\x^2},{-\k*\x^3}) node[right,black] {$\Gamma_{\xi_1}$};
\draw[ultra thick, color=blue!80,domain=0:0.6,samples=50,smooth] plot ({\x^2},{-\x^3}) node[right,black] {$\Gamma_{\xi_2}$};
\draw[ultra thick, color=blue!80,domain=0:0.6,samples=50,smooth] plot ({\x^2},{\x^3}) node[right,black] {$\Gamma_{\xi_3}$};
\draw[ultra thick, color=blue!50,domain=0:0.6,samples=50,smooth] plot ({\x^2},{\k*\x^3}) node[right,black] {$\Gamma_{\xi_4}$};

\draw[ultra thick, color=cyan,domain=0:0.6,samples=50,smooth] plot ({\x^3},{\x^2}) node[above,black] {$\Gamma_{\xi_5}$};

\draw[thick, dashed, color=brown,domain=0:0.6,samples=50,smooth] plot ({0.5*\x^3},{\x^2}) node[left,black] {$\Gamma_{\eta}$};
\draw[thick, dashed, color=brown,domain=0:0.6,samples=50,smooth] plot ({0.5*\x^3},{-\x^2}) node[left,black] {$\Gamma_{\bar\eta}$};

\end{tikzpicture}%

}
\caption{The $5$ real right semi-branches $\Gamma_{\xi_i}$ and the non-real semi-branches corresponding to the non-real Newton-Puiseux roots $\eta$ and $\bar \eta$ of $f$.}
 \label{fig:ThreeCusps}
\end{figure}
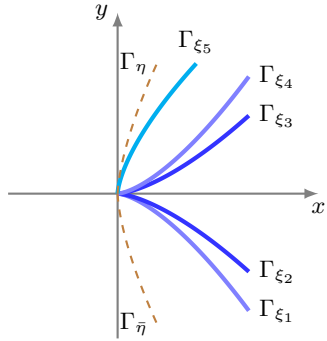

\begin{figure}[H]
 \myfigure{0.8}{%
	\begin{tikzpicture}[scale=1.2]
% define points
\path
  (0,0) coordinate(R)
  ++(2,0) coordinate(T)
  (T)
  +(2,-0.5) coordinate(V1)
  +(4,2) coordinate(V2)
  +(2,2.5) coordinate(V3)
  +(2,3.5) coordinate(V4)
  (V1)
  +(2,-2) coordinate(W1)  
  +(2,-1) coordinate(W2)    
  +(2,0.5) coordinate(W3)  
  +(2,1.5) coordinate(W4)  
;

\draw[thick]
  (R) -- (T)
  %(T) -- (V1)
  (T) -- (V1)
;
\draw[->,>=latex,thick] (T) -- (V2);
\draw[->,>=latex,thick,dotted] (T) -- (V3);
\draw[->,>=latex,thick,dotted] (T) -- (V4);
\draw[->,>=latex,thick] (V1) -- (W1);
\draw[->,>=latex,thick] (V1) -- (W2);
\draw[->,>=latex,thick] (V1) -- (W3);
\draw[->,>=latex,thick] (V1) -- (W4);

\filldraw[ultra thick,fill=white] (R) ++ (-0.1,-0.1) rectangle ++(0.2,0.2);

\path (T) node[above left]{$[x^{2/3}]$};
\path (V3) node[right]{$\eta$};
\path (V4) node[right]{$\bar{\eta}$};
\path (V1) node[below left]{$[x^{3/2}]$};

%\path (V2) node[above left]{$x^3$};

%\path (V1) node[right]{$\gamma_1$};
\path (W1) node[right]{$\xi_1$};
\path (W2) node[right]{$\xi_2$};
\path (W3) node[right]{$\xi_3$};
\path (W4) node[right]{$\xi_4$};
\path (V2) node[right]{$\xi_5$};

\path (T)--(V1) node[below,pos=0.4]{$0 x^{2/3}$};

\path (T)--(V2) node[ right,pos=0.4]{$x^{2/3}$};

\path (T)--(V3) node[ right,pos=0.6]{$j x^{2/3}$};

\path (T)--(V4) node[left,pos=0.6]{$j^2 x^{2/3}$};

\path (V1)--(W1) node[left,pos=0.6]{$ -c x^{3/2}$};
\path (V1)--(W2) node[above,pos=0.6]{$ - x^{3/2}$};
\path (V1)--(W3) node[above,pos=0.6]{$ x^{3/2}$};
\path (V1)--(W4) node[left,pos=0.6]{$ c x^{3/2}$};

\foreach \v in { T, V1}{
  \path  (\v) node[scale=4]{.};
}
\end{tikzpicture}%
}
 \caption{The real contact tree $T_{\Rr}(f)$ of the series $f$ is depicted with solid edges; the remaining edges of $T_{\Cc}(f)$ are dotted.}
 \label{fig:contactTreeInExample}
\end{figure}
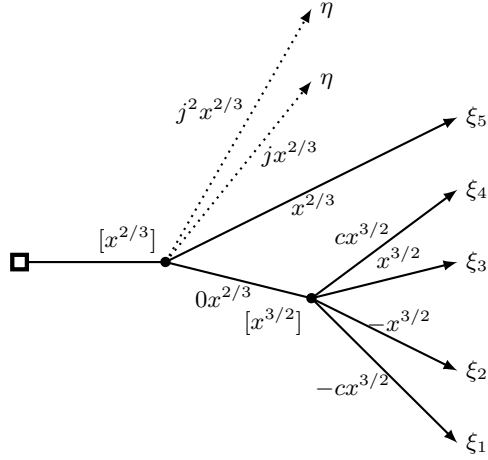

We will show that $f$ satisfies the injectivity condition \eqref{Inj} and we will apply Theorem A to compute the combinatorial type of the associated morsifications $F_{x_0}(y)$ (where $x_0 >0$ is small enough). We will see that the result depends on the parameter $c$.
The real contact tree $T_{\Rr}(f)$ is depicted with solid edges in Figure \ref{fig:contactTreeInExample}. The remaining edges of the complex contact tree $T_{\Cc}(f)$ are dotted. We have the following real total order on $\mathcal{R}_{\Rr}(f)$: 
   \[ \xi_1 \order_{\Rr} \xi_2 \order_{\Rr} \xi_3 \order_{\Rr} \xi_4 \order_{\Rr} \xi_5. \] 
We compute by termwise integration a primitive of $f$ in the sense of Equation (\ref{eq:PrimitiveOff}):
 $$ F_x(y) = \frac{1}{4} c^2 x^6 y^4 - c^2 x^8 y - \frac{1}{6} (c^2 + 1) x^3 y^6 + \frac{1}{3} (c^2 + 1) x^5 y^3 + \frac{1}{8} y^8 - \frac{1}{5} x^2 y^5. $$
This formula enables to compute the area series $S_i$ defined in Equation (\ref{eq:areaseries}), giving the following expressions:
$$
\renewcommand{\arraystretch}{1.3}
\begin{array}{l@{\hspace{7mm}}l@{\hspace{7mm}}l}
 S_1= ( \frac{2}{15}c^5 - \frac{2}{3} c^3 + \frac{2}{3} c^2 - \frac{2}{15}) x^{19/2} +\hot,
 \\
 S_2=(-\frac{4}{3} c^2 + \frac{4}{15})x^{19/2} +\hot,
 \\ 
 S_3= (\frac{2}{15} c^5 - \frac{2}{3} c^3 + \frac{2}{3}c^2- \frac{2}{15}) x^{19/2} +\hot,
 \\
 S_4=- \frac{3}{40}x^{16/3} +\hot.
 \\
\end{array}
$$

Thus $\sigma_1=\sigma_2=\sigma_3=19/2$ and $\sigma_4= 16/3$ (these exponents can also be retrieved via Proposition \ref{prop:sigmai}). We see that the initial coefficients $s_i$ are polynomials in the variable $c$. Now we may compute the series 
$\delta_i = F_x(\xi_i)$. We get:
$$
\renewcommand{\arraystretch}{1.3}
\begin{array}{l@{\hspace{7mm}}l@{\hspace{7mm}}l}
 \delta_1 = (\frac{1}{5}c^5- \frac{1}{3}(c^2 + 1)c^3+ c^3) x^{19/2} +\hot, 
 \\
 \delta_2=(c^2 - \frac{1}{3}(c^2 + 1) + \frac{1}{5}) x^{19/2} +\hot,
 \\ 
 \delta_3=(- c^2 + \frac{1}{3}(c^2 + 1) - \frac{1}{5}) x^{19/2} +\hot,
 \\
 \delta_4=(- \frac{1}{5} c^5 + \frac{1}{3}(c^2 + 1) c^3 - c^3)x^{19/2} +\hot,
 \\
 \delta_5=- \frac{3}{40}x^{16/3}+\hot.
 \\
\end{array}
$$

By our numbering, we have for small $x_0>0$: $\xi_1(x_0) < \xi_2(x_0) < \cdots < \xi_5(x_0)$. 
What is the order of the critical values $\delta_i(x_0)$?
Since $\sigma_4$ is strictly smaller than the other valuations, we have $\delta_5(x_0) < \delta_i(x_0)$, for $i=1,\ldots,4$.
Also, the relation $s_1(c)=s_3(c)$ imposes constraints on the order of the critical values. 
However, depending on $c>1$, several outcomes are still possible (see Figure \ref{fig:snakeThreeCusps}).
Let us denote, for $c>1$:
$$\lambda(c) := -\frac{s_1(c)}{s_2(c)}=\frac12\frac{(c-1)^3(c^2+3c+1)}{5c^2-1}.$$

\begin{itemize}
 \item \textbf{Case 1:} $0<\lambda(c)<\frac12$ (take for instance $c=2$). 
 In this situation $s_1+s_2 <0$, $s_2+s_3 <0$ and $s_1+s_2+s_3 <0$ so that
 for the critical values we obtain $\delta_3(x_0) < \delta_4(x_0) < \delta_1(x_0) < \delta_2(x_0)$ ($\delta_5(x_0)$ being smaller than all).
 In other words: $\xi_5 <_{\Int} \xi_3 <_{\Int} \xi_4 <_{\Int} \xi_1 <_{\Int} \xi_2$.
 The corresponding snake, i.e.{} the permutation associated to the bi-ordered critical set, is represented on the left of Figure \ref{fig:snakeThreeCusps}.
 
\begin{figure}[H]
\small
 \myfigure{0.6}{
	\begin{tikzpicture}[scale=1]

% Axes
\draw[->,>=latex,very thick, gray] (-0.5,0)--(6,0) node[below,black] {$y$};
\draw[->,>=latex,very thick, gray] (0,-0.5)--(0,6) node[left,black] {$z= F_x(y)$};

\draw[very thick, gray] (0,0) grid ++(5,5);

\begin{scope}[xshift=-0.5cm,yshift=-0.5cm]
\coordinate (p1) at (1,4);   % directly from permutation
\coordinate (p2) at (2,5);
\coordinate (p3) at (3,2);
\coordinate (p4) at (4,3);
\coordinate (p5) at (5,1);
\coordinate (p0) at (0.5,4.55);
\coordinate (p6) at (5.5,2);

\end{scope}

%\draw[very thick, red]  plot [smooth, tension=0.5] coordinates {(p0) (p1) (p2) (p3) (p4) (p5) (p6)};
\draw[very thick, green!70!black]  
(p0) ..controls +(-90:0.25) and +(180:0.25)..  (p1) 
    ..controls +(0:0.5) and +(180:0.5).. (p2) 
     ..controls +(0:0.5) and +(180:0.5)..  (p3) 
       ..controls +(0:0.5) and +(180:0.5)..  (p4) 
       ..controls +(0:0.5) and +(180:0.5)..  (p5) 
      ..controls +(0:0.5) and +(90:0.25).. (p6)
;

\foreach \i in{1,...,5}{
 \fill[orange] (p\i) circle(3pt);
 \node[blue,below] at (\i-0.5,0) {$\i$};
 \node[green!70!black,left] at (0,\i-0.5) {$\i$};
}

\node at (2.5,-1) {\bf Case 1};

\end{tikzpicture}%
\qquad
	\begin{tikzpicture}[scale=1]
	
% Axes
\draw[->,>=latex,very thick, gray] (-0.5,0)--(6,0) node[below,black] {$y$};
\draw[->,>=latex,very thick, gray] (0,-0.5)--(0,6) node[left,black] {$z$};
\draw[very thick, gray] (0,0) grid ++(5,5);

\begin{scope}[xshift=-0.5cm,yshift=-0.5cm]
  \coordinate (p1) at (1,3);   % directly from permutation
  \coordinate (p2) at (2,5);
  \coordinate (p3) at (3,2);
  \coordinate (p4) at (4,4);
  \coordinate (p5) at (5,1);
  \coordinate (p0) at (0.5,3.75);
  \coordinate (p6) at (5.5,2);
\end{scope}

%\draw[very thick, red]  plot [smooth, tension=0.5] coordinates {(p0) (p1) (p2) (p3) (p4) (p5) (p6)};
\draw[very thick, green!70!black]  
(p0) ..controls +(-90:0.25) and +(180:0.25)..  (p1) 
    ..controls +(0:0.5) and +(180:0.5).. (p2) 
     ..controls +(0:0.5) and +(180:0.5)..  (p3) 
       ..controls +(0:0.5) and +(180:0.5)..  (p4) 
       ..controls +(0:0.5) and +(180:0.5)..  (p5) 
      ..controls +(0:0.5) and +(90:0.25).. (p6)
;

\foreach \i in{1,...,5}{
 \fill[orange] (p\i) circle(3pt);
 \node[blue,below] at (\i-0.5,0) {$\i$};
 \node[green!70!black,left] at (0,\i-0.5) {$\i$};
}

\node at (2.5,-1) {\bf Case 2};
\end{tikzpicture}%
\qquad
	\begin{tikzpicture}[scale=1]
	
% Axes
\draw[->,>=latex,very thick, gray] (-0.5,0)--(6,0) node[below,black] {$y$};
\draw[->,>=latex,very thick, gray] (0,-0.5)--(0,6) node[left,black] {$z$};
\draw[very thick, gray] (0,0) grid ++(5,5);

\begin{scope}[xshift=-0.5cm,yshift=-0.5cm]
  \coordinate (p1) at (1,2);   % directly from permutation
  \coordinate (p2) at (2,4);
  \coordinate (p3) at (3,3);
  \coordinate (p4) at (4,5);
  \coordinate (p5) at (5,1);
  \coordinate (p0) at (0.5,2.75);
  \coordinate (p6) at (5.5,2);
\end{scope}

%\draw[very thick, red]  plot [smooth, tension=0.5] coordinates {(p0) (p1) (p2) (p3) (p4) (p5) (p6)};
\draw[very thick, green!70!black]  
(p0) ..controls +(-90:0.25) and +(180:0.25)..  (p1) 
    ..controls +(0:0.5) and +(180:0.5).. (p2) 
     ..controls +(0:0.5) and +(180:0.5)..  (p3) 
       ..controls +(0:0.5) and +(180:0.5)..  (p4) 
       ..controls +(0:0.5) and +(180:0.5)..  (p5) 
      ..controls +(0:0.5) and +(90:0.25).. (p6)
;

\foreach \i in{1,...,5}{
 \fill[orange] (p\i) circle(3pt);
 \node[blue,below] at (\i-0.5,0) {$\i$};
 \node[green!70!black,left] at (0,\i-0.5) {$\i$};
}
\node at (2.5,-1) {\bf Case 3};
\end{tikzpicture}%

 } 
 \caption{Depending on the value of the parameter $c>1$, several configurations of the critical graph of $F_{x_0}(y)$ are possible, giving rise to distinct snakes.}
 \label{fig:snakeThreeCusps}
\end{figure}
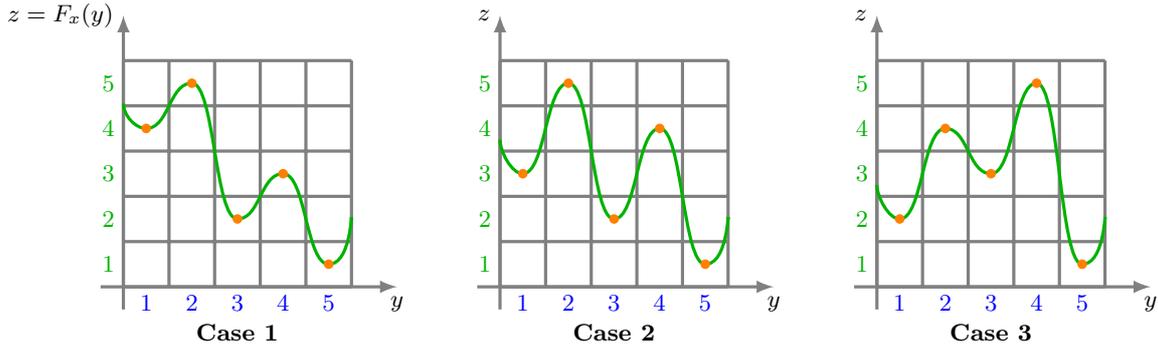

 \item \textbf{Case 2:} $\frac12 <\lambda(c) < 1$ (take for instance $c=\frac52$). 
 In this situation $s_1+s_2 <0$, $s_2+s_3 <0$ and $s_1+s_2+s_3 >0$ so that
 for the critical values we obtain $\delta_3(x_0) < \delta_1(x_0) < \delta_4(x_0) < \delta_2(x_0)$.
 The associated snake is $\pi_{F_{x_0}(y)} = \begin{pmatrix}
			{\color{blue} 1 } & {\color{blue} 2} & {\color{blue} 3} & {\color{blue} 4} & 
 {\color{blue} 5} \\
 {\color{green!70!black} 3 } & {\color{green!70!black} 5 } & {\color{green!70!black} 2 } & {\color{green!70!black} 4 } & {\color{green!70!black} 1 } \\
		\end{pmatrix}.$ Note that in this case, the permutation is non-separable (for the definition of separability, see \cite[page 13]{ghys_promenade}), whereas in \cite{sorea_portugaliae} only separable permutations were realized.

 \item \textbf{Case 3:} $\lambda(c) > 1$ (take for instance $c=3$). 
 In this situation $s_1+s_2 >0$, $s_2+s_3 >0$ and $s_1+s_2+s_3 >0$ so that
 for the critical values we obtain $\delta_1(x_0) < \delta_3(x_0) < \delta_2(x_0) < \delta_4(x_0)$.
\end{itemize}

%%%%%%%%%%%%%%%%%%%%%%%%%%%%%%%%%%%%%%%%%%%%%%%%
\bibliographystyle{plain}
\bibliography{morse06.bib}

\end{document}